\documentclass[9pt,epsfig]{amsart}
\usepackage{amssymb, adjustbox, enumerate, amsbsy, stmaryrd}
\usepackage[dvipsnames]{xcolor}
\usepackage{amsmath,wasysym}
\usepackage{comment}
\usepackage{caption}
\usepackage[mathscr]{eucal}
\usepackage{amsthm}
\usepackage{graphicx}
\usepackage {amscd}
\usepackage {epic}
\usepackage[all,color]{xy}
\usepackage[alphabetic]{amsrefs}
\usepackage{color}

\usepackage{graphpap, color}

\usepackage[mathscr]{eucal}
\usepackage{mathrsfs}
\usepackage{graphicx}

\usepackage{tikz}
\usepackage{tikz-3dplot}

\newtheorem{prop}{Proposition}[section]

\newtheorem{coro}[prop]{Corollary}
\newtheorem{defi}[prop]{Definition}

\newtheorem{exam}[prop]{Example}

\newtheorem{lemm}[prop]{Lemma}
\newtheorem{pf-thm}[prop]{proof of theorem}
\newtheorem{rema}[prop]{Remark}

\newtheorem{theo}[prop]{Theorem}

\newtheorem*{ack}{Acknowledgments}

\newcommand{\CC}{\mathbb{C}}

\newcommand{\NN}{\mathbb{N}}
\newcommand{\PP}{\mathbb{P}}

\newcommand{\RR}{\mathbb{R}}
\newcommand{\ZZ}{\mathbb{Z}}

\def\dd{\mathrm{partial}}

\def\dd{\partial}

\newcommand{\red}{\color{red}}

\newcommand{\tri}{\triangle}

\input{xy}
\xyoption{all}

\begin{document}

\title{Asymptotic Chow stability of symmetric reflexive toric varieties}

\author{King-Leung Lee}
\address{Instituto de Ciencias Matem\'{a}ticas, ICMAT\\
           C. Nicolás Cabrera, 13-15, 28049 Madrid\\ Spain}
\email{king.lee@icmat.es}

\date{\today}
\maketitle
\begin{abstract}
In this note, we study the asymptotic Chow  stability of reflexive toric varieties. We provide examples of symmetric reflexive toric varieties that are not asymptotic Chow semistable. On the other hand, we also show that any weakly symmetric reflexive toric varieties which have regular triangulation (special) are  asymptotic Chow polystable. \\

After that, we provide other criteria that can show a symmetric reflexive toric variety is
asymptotic Chow polystable. In particular, we give two examples that are asymptotic Chow  polystable, but not special. We also provide some examples of special polytopes, mainly in 2 or 3 dimensions, and some in higher dimensions.
\end{abstract}
\section{introduction}
In GIT theory, in the construction the Moduli space, we focus on those varieties that are  asymptotic Chow semistable (\cite{MFK94},\cite{GKZ94}). Also, Chow stability has many relations with other stabilities in K\"{a}hler geometry (See \cite{RT07}, \cite{Yot17} for example), so it is important to study the Chow  stability of singular varieties. However, unlike the smooth case, which is related to the constant scalar curvature manifolds (\cite{Don01}, \cite{Mab04}, \cite{Mab06}, \cite{RT07}, see also the survey paper \cite{PS09} for example), in general, K stablity or existance of cscK cannot imply Chow stability. Moreover, there is only very few examples of Chow polystable of singular varieties. In general, it is very difficult to show that a variety is  asymptotic Chow semistable. However, by the work of Futaki \cite{Fut04} and Ono \cite{Ono13}, we can determine the  asymptotic Chow polystability of toric varieties. We first recall the main theorem we used in \cite{Ono13}   (see also  \cite{LLSW19}).\\

\begin{theo}[\cite{Ono13}]
Let $P$ be a integral convex polytope of a toric variety $X_P$, and let $G<SL(n,\ZZ)$ be the biggest finite  group acting on $P$ by multiplication.  A $n$ dimensional toric variety $X_P$ is asymptotic Chow semistable iff for any $k\in \NN$, and for any convex $G$ invariant function $f$ on $kP$, we have  
\[\dfrac{1}{Vol(kP)}\int_{kP} f dV\leq \dfrac{1}{\chi(kP \cap \ZZ^n)} \sum_{v\in kP \cap \ZZ^n} f(v).\] 
\end{theo}
(As a remark, in the original literature, he used concave functions instead of convex functions, so the direction of inequality in this note is different.)\\


In this note, we mainly focus on symmetric reflexive toric varieties. One of the reasons is inspired by \cite{BS99}, which show that if a polytope is symmetric and reflexive, then it admits a K\"{a}hler Einstein metrics. With the result of \cite{Don02}, we can see that symmetric and reflexive implies K stability. So it is natural to ask if it is true for Chow stability. The second reason is, in this note, we defined a invariant called Chow-Futaki invariant, which is 
\[CF_P(a,k):=\dfrac{1}{\chi(kP \cap \ZZ^n)} \sum_{k p\in kP \cap \ZZ^n} a(p)-\dfrac{1}{Vol(P)}\int_P a(x) dV.\] As a rephrase of the corollary 4.7 in \cite{Ono13}, we can see that if $P$ is asymptotic Chow semistable, then this invariant will vanish for all $k>> k_0$ and for all affine function $a$. We can see that symmetric polytopes satisfy this criteria, so it is natural to study symmetric polytopes. Also,   by claim 4.3. in \cite{LLSW19}, there is an example which a symmetric non reflexive polytope is not asymptotic Chow  semistable.  On the other hand, by the results in \cite{LLSW19}, with the fact that $\PP^2$ and $\PP^1 \times \PP^1$ are asymptotic Chow  polystable, we can see that all 2 dimensional symmetric reflexive toric varieties are asymptotic Chow polystable. So this is natural to study symmetric reflexive polytopes.\\

However, in general, symmetric and reflexive is not enough, and a counter example is given by example \ref{symmetric reflexive non stable polytopes}. Notice that this is not an isolated example. Indeed we can construct many examples using proposition \ref{not stable}.\\

 Therefore, to ensure asymptotic Chow  polystabilities, we need more conditions on symmetric reflexive polytopes. One of the sufficient conditions is given by the following:
\begin{defi}[Definition \ref{regular triangulation}]
Let $P$ be an n dimensional integral convex polytopes on $\RR^n$. We say $P$ has regular boundary if for any $k$, there exists a triangulation of $\dd kP$ which every "triangle" is isomorphic (up to a translation and a subgroup of $SL(n-1, \ZZ)<SL(n, \ZZ)$) to  \[T_{n-1} := conv\{(0,...,0), e_1,...,e_{n-1}\},\]  the standard $(n-1)$ dimensional simplex, (i.e., the intersection between different $T_{n-1}^i$ are at the boundary) such that for any point $p\in \dd kP$, the number of simplex intersection with $p$ $\leq n!$, and this is the sub-triangulation of each face.
\end{defi}
Also, we define:
\begin{defi}[Definition \ref{special polytope}]
an integral convex polytopes on $\RR^n$ is called special if it is  reflexive, weakly symmetric and have regular boundary.
\end{defi}

And one of our main theorem is given by: 
\begin{theo}[Theorem \ref{special polytope is stable}]
Let $P$ be a special polytope, then $P$ is asymptotic chow polystable.
\end{theo}

Notice that this condition is not necessary, as theorem \ref{general condition of Chow stable} gives another sufficient criteria to show when a toric variety $P$ is asymptotic chow polystable. The statement of the theorem is the following:
\begin{theo}[Theorem \ref{general condition of Chow stable}]
Let $P$ be a integral polytope which $0\in P^0$ such that all the Futaki-Ono invariant vainish, and we have a triangulation on $kP$ by $n$ simplexes and a triangulation on  $\dd kP$ by (n-1) simplexes, we denote
$n(p;k)$ be the number of n simplex attach the $p\in kP$ under the first triangulation, and $m(p;k)$ be the  number of (n-1) simplex attach to $p\in \dd kP$ under the second triangulation. 

Suppose $n(p;k)\leq (n+1)!$ for all $p\neq 0$ and \[ \left(\dfrac{n}{2}\right)m(p;k)< ((n+1)!-n(p;k))\] for all $k$ large and for all $p\in \dd kP$, then $P$ is  asymptotic Chow polystable.
\end{theo}

As a concrete example, we have:
\begin{coro}[Corollary \ref{D(X_i) are Chow stable}]
$D(X_8)$ and $D(X_9)$ are asymptotic Chow polystable.
\end{coro}
This example shows that there are non special symmetric reflexive polytopes that are asymptotic Chow polystable.\\

In the last section, we provide examples which are asymptotic Chow polystable, mainly on dimension 3, and have two class of examples for higher dimensional. Notice that besides $D(X_8)$ and $D(X_9)$, the remaining examples are special.  Also, the corresponding varieties of the examples are given in the appendix. 

\begin{ack}
First, the author appreciate Universidad aut\'{o}noma de Madrid (UAM) and Instituto de Ciencias Matem\'{a}ticas (ICMAT) to provide a excellent environment for the author to focus on the research, and the current advisor Mario Garcia-Fernandez and his post-doc and students for having a meaning discussion and advise. Also, the author would like to thank Jacob Sturm and Xiaowei Wang for suggesting the author to start this topic, and give suggestion on the further research which yield this research. The last but not least, the author would like to thank the classmates from the Chinese University of Hong Kong to discuss mathematics and giving valuable advises.\\

The research is funded by Grant EUR2020-112265 funded by MCIN/AEI/10.13039/501100011033 and by 
the European Union NextGenerationEU/PRTR, and Grant CEX2019-000904-S 
funded by MCIN/AEI/10.13039/501100011033.
\end{ack}

\section{Chow stability of toric variety and criteria}
Recall that by Fataki, Ono (\cite{Fut04}, \cite{Ono13}, \cite{OSY12}, and also see \cite{LLSW19}), a toric variety $X_P$ is Asymptotic Chow semistable if there exists $C$ such that for any $k\geq C$, and for any convex $G$ invariant function $f$ for the corresponding polytopes $kP$, we have 
\begin{equation}\label{Chow inequality}
\dfrac{1}{Vol(kP)}\int_{kP} f dV\leq \dfrac{1}{\chi(kP \cap \ZZ^n)} \sum_{k p\in kP \cap \ZZ^n} f(p),\end{equation} and $X_P$ is polystable if the equality holds only when $v$ is affine. (In \cite{Ono13} and \cite{LLSW19}, the inequality is in opposite side as the input are concave functions.) Here $G<SL(n, \ZZ)$ is the biggest group fixing $P$, which is a discrete group.

Notice that if there exists a $\CC^*$ action acting on $X_P$, the toric variety corresponding to $P$, then it corresponds to an affine function on $P$ (See Donaldson toric variety). So we can define 
\begin{defi}\label{Chow-Futaki invariant}
Let $P$ be a integral convex polytopes. Then we define Futaki-Ono invariant of an affine function $v(x)=a_1x_1+\cdots +a_nx_n+a_0$ is given by 
\[FO_P(a,k):=\dfrac{1}{\chi(kP \cap \ZZ^n)} \sum_{k p\in kP \cap \ZZ^n} a(p)-\dfrac{1}{Vol(P)}\int_P a(x) dV.\]
\end{defi} 
We can rephrase corollary 4.7 in \cite{Ono13} as the following:
\begin{lemm}[corollary 4.7, \cite{Ono13}; also \cite{Fut04}]
Suppose $P$ is Asymptotic Chow semistable. Then there exists $C$ such that for any $k\geq C$, and for any affine function $a$ on $kP$, we have 
\[FO_P(a,k)=0.\]
\end{lemm}
Recall that
\begin{defi}
An integral convex  polytopes $P$ is symmetric if there is exactly one fix point (which must be 0) of the symmetric group $G< SL(n, \ZZ)$ acting on $P$. 
\end{defi}
In particular, any $G$ invariant affine function on a symmetric polytopes must be constant, hence it must vanish.  so we define the following: 
\begin{defi}
A polytopes $P$ is weakly symmetric if for any $k$, and for any affine function $a$ on $kP$, 
\[FO_P(a,k)=0.\] 
\end{defi}

\begin{rema}
Notice that this condition is stronger than assume $FO_P(a,k)=0$ for all $k$ large. There are two questions arise. 
\begin{enumerate}
\item It is easy to see that $P$ is symmetric implies $P$ is weakly symmetric. But is the opposite true? 
\item If $P$ is not weakly symmetric, does this imply $P$ is not asymptotic Chow semistable? 
\end{enumerate}
Notice that the K stability version is not true, as there are non symmetric K stable toric variety, namely, the toric Del Pezzo surface of degree 1.  However, it is not weakly symmetric and not  asymptotic Chow semistable (See \cite{LLSW19}, the section of $X_1$).
\end{rema}
\begin{lemm}
An weakly  symmetric integral polytopes $P$ is  (asymptotic ) Chow semistable if for any $k\in \NN$ ($k \geq C$ for some fix $C$), and for any convex function $f$ which $\displaystyle \min_{x\in kP}f(x)=f(0)=0$, we have 
\[\dfrac{1}{Vol(kP)}\int_{kP} f dV\leq \dfrac{1}{\chi(kP \cap \ZZ^n)} \sum_{k v\in kP \cap \ZZ^n} f(v),\] 
\end{lemm}
\begin{proof}
For any convex function $f$, there exists an affine function $a_k$ such that \[\min_{x\in kP}(f(x)-a(x))=f(0)=0.\] Now, 
\[\dfrac{1}{\chi(kP \cap \ZZ^n)} \sum_{k p\in kP \cap \ZZ^n} f(p)-\dfrac{1}{Vol(P)}\int_P f(x) dV=\dfrac{1}{\chi(kP \cap \ZZ^n)} \sum_{k p\in kP \cap \ZZ^n} (f-a)(p)-\dfrac{1}{Vol(P)}\int_P (f-a)(x) dV.\] Result follows.
\end{proof}
\section{Some special class of toric varieties}
\subsection{Product class}
The first class of polytopes is in the form of $P_1\times \cdots \times P_r$, where $P_1,\cdots, P_r$ and Chow stable.
\begin{lemm}\label{product convex function}
Let $P_1$, $P_2$ are bounded convex sets. Then for any f is a convex function on $P_1\times P_2$,  
$f_{P_2}(x):=\int_{P_2}f(x,y)dV_y$ is a convex function on $P_1$
\end{lemm} 
\begin{proof}
Consider $f_{P_2}(tx_1+(1-t)x_2)$, where $0\leq t\leq 1$. 
\begin{align*}
f_{P_2}(tx_1+(1-t)x_2)=&\int_{P_2}f(tx_1+(1-t)x_2,y)dV_y\\
\leq &\int_{P_2}tf(tx_1,y)dV_y+\int_{P_2}f((1-t)x_2,y)dV_y\\
=& tf_{P_2}(x_1)+(1-t)f_{P_2}(x_2).
\end{align*}
\end{proof}
\begin{prop}\label{product of chow stable polytopes}
Let $P_1$, $P_2$ are integral convex polytopes. $P_1\times P_2$ is (asymptotic) chow polystable (semistable) iff $P_1$ and $P_2$ are (asymptotic) Chow polystable (semistable). 
\end{prop}
\begin{proof}
Suppose for any $k\geq C_1$ and  $k \geq C_2$ and for any convex function $f_1, f_2$ on $P_1$ and $P_2$, we have
\[\dfrac{1}{Vol(kP_1)}\int_{kP_1}f_1(x)dV\leq \dfrac{1}{\chi(kP_1)}\sum_{p\in P_1}f_2(p);\] 
\[\dfrac{1}{Vol(kP_2)}\int_{kP_2}f_2(x)dV\leq \dfrac{1}{\chi(kP_2)}\sum_{p\in P_2}f_2(p).\] Then for any $k\geq \max\{C_1,C_2\}$ , and for any convex function $f$, we have
\begin{align*}
&\dfrac{1}{Vol(kP_1\times kP_2)}\int_{kP_1\times kP_2}f(x,y)dV_xdV_y\\
=&\dfrac{1}{Vol(kP_1)}\int_{kP_1} \dfrac{1}{Vol(kV_2)}\int_{kP_2}f(x,y)dV_ydV_x\\
=&\dfrac{1}{Vol(kP_1)}\int_{kP_1} \dfrac{1}{Vol(kV_2)} f_{kP_2}(x)dV_x \qquad (\text{lemma }\ref{product convex function})\\
\leq &\dfrac{1}{Vol(kP_2)} \dfrac{1}{\chi(kP_1)}\sum_{p_1\in kP_1\cap \ZZ^{n_1}}f_{P_2}(p_1)\\
=&\dfrac{1}{\chi(kP_1)}\sum_{p_1\in kP_1\cap \ZZ^{n_1}}\left(\dfrac{1}{Vol(kP_2)}\int_{kP_2}f(p_1,y)dV_y\right)\\
\leq & \dfrac{1}{\chi(kP_1)}\sum_{p_1\in kP_1\cap \ZZ^{n_1}}\dfrac{1}{\chi(kP_2)}\sum_{p_2\in kP_2\cap \ZZ^{n_2}}f(p_1,p_2)\\
=& \dfrac{1}{\chi(k(P_1\times P_2))}\sum_{p\in k(P_1\times P_2)\cap \ZZ^{n_1}\times \ZZ^{n_2}}f(p).
\end{align*} 
In particular, if $C_1=C_2=1$, then this inequality holds for any convex function and any $k$.\\

For the opposite, without loss of generality, assume $P_1$ is unstable. Then there exists a sequence of convex functions $f_k$ on $kP_1$ such that for any $k$ large, 
\[\dfrac{1}{Vol(kP_1)}\int_{kP_1}f_k(x)dV \geq \dfrac{1}{\chi(kP_1)}\sum_{p\in kP_1\cap \ZZ^{n_1}}f(p).\] Then define $f_k:kP_1\times kP_2\rightarrow \RR$ such that
\[f_k(x,y)=f_k(x).\] Then 
\begin{align*}
&\dfrac{1}{Vol(kP_1\times kP_2)}\int_{kP_1\times kP_2}f_k(x,y)dV=\dfrac{1}{Vol (kP_1)} \int_{kP_1}f_k(x)dv \geq   \dfrac{1}{\chi(kP_1)}\sum_{p\in kP_1\cap \ZZ^{n_1}}f_k(p)\\
=& \dfrac{1}{\chi(kP_1\times kP_2)}\sum_{p\in kP_1\cap \ZZ^{n_1}} \chi(kP_2)f_k(p)=\dfrac{1}{\chi(k(P_1\times P_2))}\sum_{p\in k(P_1\times P_2)\cap \ZZ^{n_1}\times \ZZ^{n_2}} f_k(p). \end{align*}
\end{proof}

The following class of polytopes may not be chow stable in general. Indeed, we will give a criteria which this class must not be asymptotic Chow polystable.

As a quick check, we have a computational proof of the following well known fact:
\begin{coro}
$((\PP^1)^n, -K_{(\PP^1)^n})$ is asymptotic Chow polystable.
\end{coro}
\begin{proof}
$[-1,1]$ is asymptotic Chow polystable. A direct consequence of proposition \ref{product of chow stable polytopes} implies $[-1,1]^n$ is asymptotic Chow polystable.
\end{proof}
\subsection{Symmetric Double cone type}
We now consider a class of examples that can construct unstable polytopes. Also, it is one of the non trival and easiest class to study.
\begin{defi}
Let $P$ be a n dimensional integral polytopes. Then we define the double cone
\[D(P):=Conv\{0,...,0,1), (0,...,0,-1), (p,0)| p\in P\}.\]
\end{defi}
Notice that 
\[kD(P)=\{(p, q)\in \RR^n\times \RR| p\in (k-q)P, -k\leq q\leq k\}.\]

\begin{lemm}
Suppose $P$ is symmetric, then $D(P)$ is symmetric.
\end{lemm}
\begin{proof}
If $G$ is the largest group acting on $P$, then $G\times \ZZ/ 2\ZZ$ acting on $D(P)$ by 
\[(g, \pm 1)\cdot (p,q)=(g\cdot p, \pm q). \] hence if $P$ is symmetric, then $D(P)$ is symmetric.
\end{proof}

To given a counter example, first we have the following well known fact.
\begin{lemm}[See \cite{Ehr77}, or  \cite{BDDPS05}]\label{number of integral points}
Let $P$ be a convex integral polytope with dim $\geq 2$. then the number of point 
\[\chi(kP):=| kP\cap \ZZ^n|=Vol(P)k^n+\dfrac{1}{2}Vol( \dd P)k^{n-1}+p(k),\] where $p(k)$ is a polynomial in $k$ of degree $n-2$ which depends on $P$ only.
And for $n=1$, 
\[\chi(kP)=Vol(P)k+1;\]
for $n=2$, we have the Pick theorem (see \cite{Pic1899}):
\[\chi(kP):=| kP\cap \ZZ^n|=Vol(P)k^2+\dfrac{1}{2}Vol( \dd P)k+1.\]

In particular, for $k$ large,  \[\chi(kP)-Vol(kP)=\dfrac{Vol(\dd P)}{2}k^{n-1}+p(k)> 0.\]
\end{lemm}
\begin{prop}\label{not stable}
Let $P$ be a n dimensional integral polytopes. Suppose $Vol(P)\geq  (n+2)(n+1)$, then $D(P)$ is not asymptotic Chow semistable.
\end{prop}
\begin{proof}
For $kD(P)$, denote the point in $kD(P)$ be $(p,q)$, where $p\in \RR^n$, $q\in \RR$. Consider the following function: 
\[f(p,q)=\left\{ \begin{matrix} 0 & \text{ if } & |q|\leq k-1\\
 t & \text { if } & |q|= (1-t)(k-1)+tk =k-1+t,  0\leq t\leq 1
\end{matrix}\right.\]
Then 
\[\sum_{p\in kD(P)}f(p)=2.\]
Let $Vol(P)= (n+2)(n+1)(1+\delta)$ for some $\delta \geq 0$ .
\begin{align*} &\int_{kD(P)} f(x) dV= 2\int_{0}^1 t(1-t)^nVol(P) dt=2Vol(P)\int_{0}^1 t^n(1-t)dt=2Vol(P)\left(\dfrac{1}{n+1}-\dfrac{1}{n+2}\right) \\
=& 2 \dfrac{Vol(P)}{(n+1)(n+2)}=2+2\delta\end{align*} for some fix $\delta>0$. Therefore, 
\[\dfrac{1}{Vol(kD(P))}\int_{kD(P)} f(x) dV=\dfrac{2+2\delta}{Vol(D(P))k^{n+1}}\] and 
\[ \dfrac{1}{\chi(kD(P))}\sum_{p\in kD(P)}f(p)=\dfrac{2}{\chi(kD(P))}.\] 
\begin{align*}
&\dfrac{1}{\chi(kD(P))}\sum_{p\in kD(P)}f(p)- \dfrac{1}{Vol(kD(P))}\int_{kD(P)} f(x) dV\\
 =& \dfrac{2}{\chi(kD(P))}- \dfrac{2+2\delta}{Vol(kD(P))}\\
<&  \dfrac{2}{Vol(kD(P))}- \dfrac{2+2\delta}{Vol(kD(P))}\\
=& \dfrac{-2\delta}{Vol(k D(P))}\\
\leq & 0.
\end{align*}
\end{proof}
\begin{exam}[Claim 4.3 in \cite{LLSW19}]
Let $P=[-a,a]$ for $a> 3$, then $D(P)$ is not Chow stable by previous proposition.
\end{exam}

The question is, suppose $P$ is symmetric and reflexive, is it enough to show that $P$ is asymptotic Chow semistable? The answer is no. 
\begin{exam}\label{symmetric reflexive non stable polytopes}
Consider $P=[-1,1]^6= ((\PP^1)^6, O(2,2,2,2,2,2)) $, then \[Vol(P)=2^6=64>56=8\times 7=(6+2)(6+1). \]
Indeed, as $2^x-(x+2)(x+1)$ is increasing when $x\geq 6$, so for all $n\geq 6$, 
\[2^n-(n-2)(n-1)\geq 64-56=8>0,\]
which implies that $D([-1,1]^n)$ is not  asymptotic Chow semistable for all $n\geq 6$.
\end{exam}

\begin{rema}
This cut a vertex technique  obviously holds for any polytopes. Namely, let $p$ be a $d$ dimensional polytopes, then we cut a cone from the vertex such that there is no interior integral point until length 1, and let the base to be $Q_p$, which is $(d-1)$ dimension.  If $Vol(Q_p)\geq d(d+1)$, then $P$ is not asymptotic Chow stable. \\

Beside, for $(\PP^{n}, O(n+1))$ and $((\PP^1)^{n}, O(2,\cdots, 2))$, if we cut the cone, it must be a $n$ dimensional simplex,  hence $Q_p$ is a $n-1$ dimensional simplex for all $p$, and the volume of $Q_p$ is \[Vol(Q_p)\dfrac{1}{(n)!}<(n+2)(n+1),\] which is expected as we know that they are asymptotic Chow polystable .   
\end{rema}

In the next section, we will define a more restrictive type of polytopes, which is asymptotic Chow polystable.
\section{special polytopes}
We first recall some definition in toric geometry.
\begin{defi}
A integral polytopes $P$ is reflexive if the boundary is given by the equations \[\sum_{i=1}^n a_ix_i=\pm 1,\] where $a_i\in \ZZ$. Or equialvently, there exists exactly one interior point $(0,...,0)$.
\end{defi}
\begin{defi}
A integral polytopes $P$ is symmetric if there is exactly one fix point  of the symmetric group $G$ acting on $P$. 
\end{defi} 
notice that if $P$ is reflexive, then the fix point is $0$, and $G< SL(n, \ZZ)$ acting on $P$ as a multiplication.
We now add one extra restriction on the symmetric reflexive polytopes.
\begin{defi}\label{regular triangulation}
Let $P$ be an n dimensional integral convex polytope on $\RR^n$. We say $P$ has a regular boundary if for any $k$, there exists a triangulation of $\dd kP$ which every "triangle" is integrally isomorphic of  \[T_{n-1} := conv\{(0,...,0), e_1,...,e_{n-1}\},\]  the standard $(n-1)$ dimensional simplex, (i.e., the interseciton between different $T_{n-1}^i$ are in the boundary) such that for any point $p\in \dd kP$, the number of simplex intersection with $p$, denoted as $n(p)$, satisfies \[n(p)\leq n!,\] and this triangulation is a sub-triangulation of each face. 

Here integrally isomorphic means one of the object is obtained from another object by an integral rigid motion,  i.e., the multiplication of a matrix $A\in SL(n, \ZZ)$ and translation of $v\in \ZZ^{n}$.  
\end{defi}

\begin{rema}
 If two object $P_1, P_2$ are integral isomorphic , then for all $k$, $kP_1$ has same number of integral points as $kP_2$. Indeed, integral isomorphism is obtained by a bijection map $\varphi:\ZZ^n\rightarrow \ZZ^n$,. So for each compact object $U\subset \RR^n$, the map $\varphi: U\cap \ZZ^n\rightarrow \varphi(U\cap \ZZ^n)$ is a bijection. 
\end{rema} 
\begin{defi}\label{special polytope}
An integral convex polytope on $\RR^n$ is called special if it is  reflexive, weakly symmetric and has regular boundary.
\end{defi}

\begin{exam}\label{2 dimension example}
Suppose $P$ is a two dimensional symmetric reflexive polytope, then it is special. 
\end{exam}
\begin{proof}
The boundary of $P$ is a loop, hence every point must connect with 2 segment, hence the boundary has regular triangulation. 
\end{proof}
\begin{rema}\label{2 dimension example remark}
 The two dimensional symmetric reflexive polytopes are
$X_3:=Conv\{(-1,-1), (1,0), (0,1)\}$, $X_4:=Conv\{(\pm 1, 0), (0, \pm 1)\}$, $X_6:=Conv\{(0,\pm 1), (\pm 1, 0), (1,-1), (-1,1)\}$, $X_8:=Conv\{(\pm 1, \pm1)\}$, $X_9:=Conv\{(-1,-1), (2,-1), (-1, 2)\}$. 
\end{rema}
\begin{exam}
$D(X_3)$, $D(X_4)$ and $D(X_6)$ are special. However,
$D(X_8), D(X_9)$ is symmetric and reflexive only, and they are  not special. the reason is, the face of $D(X_8)$ is given by the triangle $Conv\{(-1,0),(0,0),(1,0),(0,1)\}$, which for any $k$, for the point $(0,0,\pm k) $, there must be $2$ simplex attaching the vertex for each face, and there are 4 face, hence \[n(0,0, \pm k)=8.\] 
Similarly, we can see that for any triangulation for $D(X_9)$, 
\[n(0,0, \pm k)=9\]
\end{exam}

\section{Chow stability of special polytopes}
We now consider what the individual assumptions can be provided. we first start with reflexive.
\begin{lemm}\label{all point is on boundary}
Let $P$ is a reflexive polytopes. then for all $k\in \NN$, \[kP\cap \ZZ^n=\bigcup_{i=0}^k (\dd iP \cap \ZZ^n).\]
\end{lemm}
\begin{proof}
Let $P$ be reflexive. $(0,...,0)\in \dd (0P)$ by defintion.
Notice that for any $p=(p_1,...,p_n)\neq 0 \in kP$, there exists $\alpha$ and $0<c_{\alpha}<k$ such that 
\[a_{1,\alpha}p_1+\cdots +a_{n,\alpha}p_n=c_{\alpha}.\] But $p\in \ZZ^n$ implies $c_{\alpha\in \ZZ}$, hence $p\in \dd c_{\alpha} P\cap \ZZ^n$.  
\end{proof}

Also, we have the following.
\begin{lemm}\label{boundary volume and volume}
Let $P$ be reflexive $n$ dimensional polytopes, then 
\[\dfrac{Vol(\dd P)}{n}=Vol(P)\]
\end{lemm} 
\begin{proof}
Let $\displaystyle \cup_{i=1}^r=Q_i=\dd P$, where $Q_i$ are faces of $P$. Then define 
\[C(Q_i):=Conv\{(0,...,0), Q_i\}=\{tx\in P| x\in Q_i, 0\leq t\leq 1\}.\] Then
\[P= \bigcup_{i=1}^r C(Q_i),\] and 
\[Vol(P)=\sum_{i=1}^r Vol(C(Q_i)).\] The assumption that $P$ is reflexive implies the height is 1 for any $C(Q_i)$, so 
\[Vol(P)=\sum_{i=1}^r Vol(C(Q_i))=\sum_{i=1}^r \dfrac{Vol(Q_i)}{n}=\dfrac{Vol(\dd P)}{n}.\]
\end{proof}


\begin{lemm}\label{convexity on t direction}
	Suppose  $f:P\rightarrow \RR$ be a $G$ invariant convex function such that \[\min_{p\in P}f(x)=f(0)\geq 0.\]  Then 
	\[F_f(t):=\int_{t\dd P}f(tx)d\sigma_P\] is convex, where $\sigma_{\dd P}|_x =d(l_{Q_i})|_x$ for $x\in Q_i$, the defining boundary function of the face $Q_i\subset \dd P$.  

\end{lemm}
\begin{proof}
First, we have a map $\varphi: \dd P \times [0,1]\rightarrow P$ defined by 
\[\varphi(x,t)=tx.\] Notice that this map is surjective, $\varphi(x,0)=0$ and $\varphi|_{\dd P\times (0,1]}$ is bijective. Hence any function $f$ on $P$ can be represented by the function 
\[g(x,t):=f \circ \varphi(tx)\]
	Notice that $f(0)$ is the minimum, so $f(x)\geq 0$. We find a (decreasing) sequence of smooth $G$ invariant convex function $f_i$ with $f_i(0)\geq 0$, converges to $f$. 
Denote $Q= \dd P$.
	We define $g_i:Q\times [0,1]\rightarrow \RR$ by \[g_i(x,t):=f_i\circ \varphi(x,t)).\] Now, by convexity, and $f(0)$ is minimum, $f$ is increasing along the segment $\{(tx,t)|0\leq t\leq 1\}$, so it  implies 
	\[\dfrac{dg_i}{dt}(x,t)\geq 0.\]
	Also, convexity of $f_i$ implies
	\[\dfrac{d^2g_i}{dt^2}(x,t)\geq0\]
	
	 As \[\int_{tQ}f_i(tx)d\sigma_Q=t^{n-1}\int_Q g_i\left(x,t\right)d\sigma_Q,\] We now compute the second derivative of $F_i$. for $n\geq 3$, the second derivative of $F_i$ is given by:
	\begin{align*}
	&\dfrac{d^2}{dt^2}\int_{tQ}f_i(x,t)d\sigma_Q\\
	=&\dfrac{d^2}{dt^2}t^{n-1}\int_{Q}g_i(x,t)d\sigma_Q\\
	=&\dfrac{d}{dt}\left((n-1)t^{n-2}\int_Qg_i\left(x,t\right)d\sigma_Q+t^{n-1}\int_Q\dfrac{dg_i}{dt}\left(x,t\right)d\sigma_Q\right)\\
	=&(n-1)(n-2)t^{n-3}\int_Qg_i\left(x,t\right)d\sigma_Q\\
	&+2(n-1)t^{n-2}\int_Q\dfrac{dg_i}{dt}\left(x,t\right)d\sigma_Q+t^{n-1}\int_Q \dfrac{d^2g_i}{dt^2}d\sigma_Q\\
	\geq & 0,
	\end{align*}
	so all $F_i$ are convex. Thus $F$ is convex.
	
Also, for $n=2$, 
\[F_i''(t)=2(n-1)\int_Q\dfrac{dg_i}{dt}\left(x,t\right)d\sigma_Q+t\int_Q \dfrac{d^2g_i}{dt^2}d\sigma_Q,\] Finally, for $n=1$, $F(t)=f(-ta)+f(tb)$ for $P=[-a,b]$, so 
\[F_i''(t)=a^2 f''(-ta)+b^2 f''(tb)\geq 0.\]

So $F_i''(t) \geq 0$ for all $i$ which implies $F(t)$ is convex. 
\end{proof}

As a remark, when we put $f(x)=c$, then $F_c(t)= cVol(\dd P)t^{n-1}$, in which we can see if $c<0$ and $n \geq 3$, $F_c$ is not convex on $[0,1]$.
%

\begin{coro}\label{integral estimation by boundary}
Suppose $P$ is symmetric, then for all $k\in \RR$, for all $G$ invariant convex function $f:kP\rightarrow \RR$ with $\displaystyle \min_{x\in kP}f(x)=f(0)=0$,  we have 
\[\int_{kP} f(x,t) dV\leq \dfrac{1}{2}F(0)+F(1)+...+F(k-1)+\dfrac{1}{2}F(k),\]
where \[F(t):=\int_{t \dd P}f(x,t)d\sigma_{\dd P}.\] Also, equality hold if and only if $f=0$. 
\end{coro}
\begin{proof}
Now \[\int_{kP} f(tx)dV=\int_0^1\int_{t\dd kP}f(tx)d\sigma dt=\int_0^1 F_{f, kP}(t)dt.\]
By Lemma \ref{convexity on t direction}, $F(t)$ is convex, hence by trapezoid rule, we have 
\[\int_{kP} f(x,t)\leq \dfrac{1}{2}F(0)+F(1)+...+F(k-1)+\dfrac{1}{2}F(k).\]
\end{proof}

The final lemma is a property of regular boundary:
\begin{lemm}\label{boundary estimate by discrete sum}
Let $P$ have a regular boundary, then for any $k$, and for any convex function $f$, we have
\[\int_{\dd kP}f(x)d\sigma \leq \sum_{v\in \dd kP}f(v).\]
\end{lemm}
\begin{proof}
Let $n$ be the dimension of $P$, then its boundary can be triangulated  by the (n-1) simplex $T_{n-1}$. Let the vertex point of $T_{n-1}$ $p:=(p_0,...,p_{n-1})$ and $m_p$, then convexity implies 
\[\int_{T_{n-1}}f(x)d\sigma\leq Vol(T_{n-1})\sum_{i=0}^{n-1}\dfrac{f(p_i)}{n}=\sum_{i=0}^{n-1}\dfrac{f(p_i)}{n!}. \]
Therefore, if we denote $n(p)$ to be the number of simplex touch the point $p$, then the regular boundary assumption means $n(p) \leq n!$, and hence
\begin{align*}\int_{\dd kP}f(x)d\sigma=&\sum_{\alpha} \int_{T_{n-1}^{\alpha}}f(x)d\sigma
\leq \sum_{\alpha}\sum_{i=0}^{n-1}\dfrac {f(p_i^{\alpha})}{n!}
\displaystyle=\sum_{p\in \dd kP \cap \ZZ^n}\dfrac{ n(p)f(p)}{n!}\\
&\leq\sum_{p\in \dd kP \cap \ZZ^n} \dfrac{ n!f(p)}{n!}
=\sum_{p\in \dd kP \cap \ZZ^n} f(p). 
\end{align*}
 \end{proof}
\section{Chow stabilities of special polytopes}

We now show that a special polytope is asymptotic Chow polystable.
\begin{theo}\label{special polytope is stable}
Let $P$ be a special polytope, then $P$ is asymptotic Chow polystable.
\end{theo}
\begin{proof}
First, denote $\chi(kP):=\#\{kP\cap \ZZ^n\}$, then 
\[\dfrac{1}{Vol(kP)}\int_{kP}c dV=\dfrac{1}{\chi(kP)}\sum_{p\in kP\cap \ZZ^n}c, \] so we only need the show the inequality for all $G$ invariant convex non negative function $f\geq 0$. And we can assume \[\min_{x\in kP}f(x)=f(0)\geq 0.\]
As $P$ is symmetric, for any non-negative convex function, by corollary \ref{integral estimation by boundary}
\[\int_{kP} f(x)dV\leq \dfrac{1}{2}f(0)+ \sum_{r=1}^{k-1}\int_{\dd rP} f(x)d\sigma+\dfrac{1}{2}\int_{\dd kP} f(x)d\sigma. \] 

Lemma \ref{boundary estimate by discrete sum} implies 
\[\int_{kP} f(x)dV\leq \dfrac{1}{2}f(0)+ \sum_{r=1}^{k-1}\sum_{\dd rP \cap \ZZ^{n}} f(p)+\dfrac{1}{2}\sum_{\dd kP\cap \ZZ^n} f(p)=\sum_{r=0}^k \sum_{p\in \dd rP\cap \ZZ^n}f(p)- \dfrac{1}{2}f(0)- \dfrac{1}{2}\sum_{p\in \dd kP\cap \ZZ^n}f(p). \] 
Therefore, lemma \ref{all point is on boundary} implies 
\[\int_{kP} f(x)dV\leq \sum_{kP \cap \ZZ^{n}} f(p)- \dfrac{1}{2}f(0)-\dfrac{1}{2} \sum_{p\in \dd kP\cap \ZZ^n}f(p). \] 

Therefore, 
\begin{align*}
&\dfrac{1}{Vol(kP)}\int_{kP} f(x)dV \\
=& \dfrac{1}{\chi(kP)}\int_{kP} f(x)dV +\left(\dfrac{1}{Vol(kP)}-\dfrac{1}{\chi(kP)}\right)\int_{kP} f(x)dV\\
\leq & \dfrac{1}{\chi(kP)}\left( \sum_{r=0}^{k}\int_{\dd rP} f(x)d\sigma-\dfrac{1}{2}f(0)-\dfrac{1}{2}\sum_{p\in \dd kP\cap \ZZ^n } f(p)\right)\\
&+\left(\dfrac{1}{Vol(kP)}-\dfrac{1}{\chi(kP)}\right)\int_{kP} f(x)dV\\
\leq &  \dfrac{1}{\chi(kP)} \sum_{r=0}^k \sum_{p\in \dd rP}f(p)- \dfrac{1}{2 \chi(kP)}\left( f(0)+\sum_{p\in \dd kP\cap \ZZ^n} f(p)\right)\\
&+\left(\dfrac{1}{Vol(kP)}-\dfrac{1}{\chi(kP)}\right)\int_{kP} f(x)dV \qquad & (\text{Lemma \ref{boundary estimate by discrete sum}})\\
=& \dfrac{1}{\chi(kP)}  \sum_{p\in kP}f(p)- \dfrac{1}{2 \chi(kP)}\left( f(0)+\sum_{p\in \dd kP\cap \ZZ^n} f(p)\right)\\
&+\left(\dfrac{1}{Vol(kP)}-\dfrac{1}{\chi(kP)}\right)\int_{kP} f(x)dV \qquad & (\text{Lemma \ref{all point is on boundary}})
\end{align*} 
So we only need to show 
\[- \dfrac{1}{2 \chi(kP)}\left( f(0)+\sum_{p\in \dd kP\cap \ZZ^n} f(p)\right)+\left(\dfrac{1}{Vol(kP)}-\dfrac{1}{\chi(kP)}\right)\int_{kP} f(x)dV \leq 0,\] that is, 
\begin{equation}\label{main inequality}
\left(\dfrac{1}{Vol(kP)}-\dfrac{1}{\chi(kP)}\right)\int_{kP} f(x)dV\leq \dfrac{1}{2 \chi(kP)}\left( f(0)+\sum_{p\in \dd kP\cap \ZZ^n} f(p)\right)\end{equation}
Now, we can triangulate $kP$ by $C_{\alpha}:=conv \{(0,...,0),T_{n-1}^{\alpha}\}$, wehre $\cup_{\alpha}T_{n-1}^{\alpha}$ is the regular triangulation on $\dd kP$. 
$Vol(C_{\alpha})=\dfrac{k}{n(n-1)!}=\dfrac{k}{n!}$, and by convexity, 
\begin{align*}
\int_{kP} f(x)dV  \leq & \sum_{\alpha} Vol(C_{\alpha})\sum_i \dfrac{f(0)+f(p_0^{\alpha})+\cdots +f(p_{n-1}^{\alpha})}{n+1}\\
=&\sum_{p\in \dd kP}\dfrac{k n(p) f(p)}{(n!) (n+1)}+ \dfrac{Vol(\dd kP)}{Vol(C_{n-1})} \dfrac{k}{n! (n+1)}f(0)
\leq \sum_{p\in \dd kP} \dfrac{kf(p)}{n+1}+ Vol(\dd kP) \dfrac{k}{n(n+1)}f(0)
\end{align*}

Therefore, in order to show equation (\ref{main inequality}), it suffices to show that we have
\[
\left(\dfrac{1}{Vol(kP)}-\dfrac{1}{\chi(kP)}\right)\left(\sum_{p\in \dd kP} \dfrac{kf(p)}{n+1}+ Vol(\dd kP) \dfrac{k}{(n+1)n}f(0)\right)\leq \dfrac{1}{2 \chi(kP)}\left( f(0)+\sum_{p\in \dd kP\cap \ZZ^n} f(p)\right),\] or, 
\[\left[\left(\dfrac{\chi(kP)-Vol(kP)}{Vol(kP)}\right)\left(\dfrac{k }{n(n+1)}{Vol(\dd kP)}\right)-\left(\dfrac{1}{2}\right)\right]f(0) \leq \left(\dfrac{1}{2}-\left(\dfrac{k}{n+1}\right)\left(\dfrac{\chi(kP)-Vol(kP)}{Vol(kP)}\right)\right)\sum_{p\in \dd kP} f(p)\]

By assumption, $\displaystyle f(0)=\min_{p\in kP}f(p) = 0$, so we only need to show 
\[0 \leq \left(\dfrac{1}{2}-\left(\dfrac{k}{n+1}\right)\left(\dfrac{\chi(kP)-Vol(kP)}{Vol(kP)}\right)\right) .\] 

By Lemma \ref{number of integral points}, $\chi(kP)=Vol(P)k^n+\dfrac{1}{2}Vol(\dd P)k^{n-1}+r(k)$, where  $r(k)= a_{n-2}k^{n-2}+ \cdots +a_1 k+ 1$ is a polynomial, and $a_i$ depends on $P$ only.
\[\dfrac{\chi(kP)-Vol(kP)}{Vol(kP)}= \dfrac{Vol(\dd P)}{2k Vol(P)}+r(k)k^{-n}.\] Using Lemma \ref{boundary volume and volume}, 
\begin{align*}\left(\dfrac{k}{n+1}\right)\left(\dfrac{\chi(kP)-Vol(kP)}{Vol(kP)}\right) = &\dfrac{k}{n+1}\left(\dfrac{Vol(\dd P)}{k Vol(P)}+r(k)\dfrac{k^{-n}}{Vol(P)}\right)\\
=&\dfrac{k}{n+1}\left(\dfrac{n}{k}+r(k)\dfrac{k^{1-n}}{Vol(P)}\right)\\
=&\dfrac{n}{2(n+1)}+r(k)\dfrac{k^{1-n}}{Vol(P)}.
\end{align*} Therefore, there exists $C$ such that
\[\dfrac{|r(k)k^{1-n}|}{Vol(P)}=\dfrac{1}{Vol(P)}|a_{n-2}k^{-1}+...+a_1 k^{2-n}+k^{1-n}|< \dfrac{1}{2(n+1)}\] for all $k\geq C$, and hence 
\[\left(\dfrac{k}{n+1}\right)\left(\dfrac{\chi(kP)-Vol(kP)}{Vol(kP)}\right)\leq \dfrac{n}{2(n+1)}+\dfrac{|r(k)k^{1-n}|}{Vol(P)}<\dfrac{1}{2},\] which shows our theorem. 
\end{proof}
%

\begin{exam}[See also \cite{LLSW19}]
By example \ref{2 dimension example} and remark \ref{2 dimension example remark}, all 2 dimensional  symmetric reflexive polytopes are special, which are $X_i$  for $i=3,4,6,8,9$, hence the above varieties are asymptotic chow polystable.
\end{exam}

\section{regular triangulation of $n$ simplex}
To find higher dimensional examples, we first need to know how to triangulate a polytopes in higher dimensions. In general it may be very difficult, but at least, we can triangulate a polytopes of $kP$ by the following:
\begin{enumerate}
\item triangulate $P$ into simplex;
\item For $kP$, we first enlarge the triangulation on $P$, then $kP$ is triangulated by enlarge simplexes $kT_n$, then we further triangulate every enlarged n simplex $kT_n$ into simplexes. 
\end{enumerate}
So we need to know how to triangulate a simplex $kT_n:= Conv\{(0,...,0), ke_i|i=1,...,n\}$. where $ke_1=(k,0,...,0), ..., ke_n=(0,...,0,k)$.
\begin{lemm}\label{triangulation of simplex}
There exists a triangulation of $\RR^n$ such that $n(p)=(n+1)!$ for all $p\in \ZZ^n$. Moreover, this triangulation can triangulate $kT_n$ such that 
\[n(p)=\dfrac{(n+1)!}{(k+1)!} \text{ for all }p\in ((n-k) \text{ skeleton})^o\cap \ZZ^n.\]  
\end{lemm}
\begin{proof}
Let $\alpha \in \{0,1\}^n- (0,...,0)$, and consider all the hyperspace $\alpha \cdot x =p$. Notice that any intersection of n of the hyperspace is an integral point. Then we have a triangulation of $\RR^n$, except we don't know if each "triangle" has the smallest area. To do so, notice that this triangulation is translation invariant, so it is sufficient compute $n(p)$  for $p=(0,...,0)$. \\

Let $a_0=0$, and let $(a_1,...,a_n)$ is a generic point near $(0,...,0)$. Then we have 
\[a_{\sigma(0)}>a_{\sigma(1)}> \cdots >a_{\sigma(n)},\] where $\sigma\in S_n$ is an element in the symmetric group of $\{0,...,n\}$. Hence there are $(n+1)!$ element. Notice that if $\sigma(i)>\sigma(j)$ and $\sigma'(i)<\sigma'(j)$, then the plane $x_i+x_j=0$ separate this two points.  Therefore, 
\[n(p)\geq (n+1)!.\] But $n(p)$ is a constant and the volume of any integral polytope is at least $\dfrac{1}{(n+1)!}$, hence we prove the first part. \\

For the second part, notice that the group $S_{n+1}$ acts on this triangulation in $\RR^n$. So without loss of generality, we may consider the points in  $(n-k)$ skeleton is in $a_1=\cdots =a_k=0=a_0$. hence the group fixing the points are $S_{k+1}$, which implies 
\[n(p)= \dfrac{(n+1)!}{(k+1)!}\] for $p \in ((n-k) \text{ skeleton})^o\cap \ZZ^n$. 
\end{proof}

\begin{center}
\begin{figure}[h!]
\begin{tikzpicture}[x=1cm,y=1cm]
\draw[red, thick] (0,2) -- (1,2);
\draw[red, thick] (1,2) -- (1,1);
\draw[red, thick] (0,2) -- (1,1);
\draw[red,thick] (1,1) -- (2,1);
\draw[red,thick] (0,1) -- (2,1);
\draw[red,thick] (1,0) -- (1,2);
\draw[red,thick] (2,0) -- (2,1);
\draw[red,thick] (1,1) -- (2,0);
\draw[red,thick] (0,1) -- (1,0);

\draw[gray, thick] (0,0) -- (0,3);
\draw[gray, thick] (0,0) -- (3,0);
\draw[gray, thick] (0,3) -- (3,0);
\draw[gray, thick] (0,3) -- (3,0);
\foreach \Point/\PointLabel in
{
(0,3)/,(0,2)/, (0,1)/, (0,0)/,
 (1,2)/, (1,1)/, (1,0)/,
 (2,1)/, (2,0)/, (3,0)/,
}
\draw[fill=gray] \Point circle (2pt) node[above right] {$\PointLabel$};
\foreach \Point/\PointLabel in
{
(1,1)/O,
 }
\draw[fill=black] \Point circle (2pt) node[below left] {$\PointLabel$};
\end{tikzpicture}
\begin{tikzpicture}[x=1cm,y=1cm]
\draw[red, thick] (0,1) -- (1,0);
\draw[red, thick] (0,2) -- (2,0);
\draw[red, thick] (1,2) -- (3,0);
\draw[red, thick] (2,2) -- (3,1);
\draw[red, thick] (0,1) -- (3,1);
\draw[red, thick] (1,0) -- (1,2);
\draw[red, thick] (2,0) -- (2,2);

\draw[gray, thick] (0,0) -- (3,0)--(3,2) -- (0,2) -- (0, 0);

\foreach \Point/\PointLabel in
{
(0,0)/, (1,0)/,(2,0)/, (3,0)/,
(0,1)/, (1,1)/,(2,1)/, (3,1)/,
(0,2)/, (1,2)/,(2,2)/, (3,2)/,
}
\draw[fill=gray] \Point circle (2pt) node[above right] {$\PointLabel$};

\end{tikzpicture}
\caption{triangulation of 2 simplex and rectangle. }\label{triangulation of 2 simplex}
\end{figure}
\end{center}

%

\section{A sufficient condition of Chow stabilities on toric varieties}
\begin{theo}\label{general condition of Chow stable}
Let $P$ be a integral polytope which $0\in P^0$ such that all the Futaki-Ono invariant vainish, and we have a triangulation on $kP$ by $n$ simplexes and a triangulation on  $\dd kP$ by (n-1) simplexes, we denote
$n(p;k)$ be the number of n simplex attach the $p\in kP$ under the first triangulation, and $m(p;k)$ be the  number of (n-1) simplex attach to $p\in \dd kP$ under the second triangulation. 

Suppose $n(p;k)\leq (n+1)!$ for all $p\neq 0$ and \[ \left(\dfrac{n}{2}\right)m(p;k)< ((n+1)!-n(p;k))\] for all $k$ large and for all $p\in \dd kP$, then $P$ is  asymptotic Chow polystable.
\end{theo}
\begin{proof}
First, we can write 
\[\dfrac{1}{Vol(kP)}\int_{kP} f(x) dV =\dfrac{1}{\chi(kP)}\int_{kP} f(x)dV+\left(\dfrac{1}{Vol(kP)}-\dfrac{1}{\chi(kP)}\right)\int_{kP} f(x)dV.\] 
Now, 
\[\int_{kP} f(x)dV \leq \sum_{p\in kP}\dfrac{n(p;k)f(p)}{(n+1)!}\leq  \sum_{p\in kP}f(p)-\sum_{p\in \dd kP}\dfrac{(n+1)!-n(p;k)}{(n+1)!}f(p).\]

Also, using the triangulation of $\dd kP$, and cone with origin, 
\[\int_{kP} f(p) dV \leq \sum_{p\in \dd kP}\dfrac{m(p;k)kf(p)}{(n)!(n+1)}+Vol(kP)\dfrac{f(0)}{(n+1)}.\] Also, 
\[\chi(kP)-Vol(kP)=\dfrac{Vol(\dd (P))k^{n-1}}{2}+O(k^{n-2})=\dfrac{nVol(P)k^{n-1}}{2}+O(k^{n-2})\] By symmetric, we may assume $f(0)=0$, 
\begin{align*}
&\dfrac{1}{Vol(kP)}\int_P f(x) dV\\
  \leq & \dfrac{1}{\chi(kP)} \left(\sum_{p\in kP}f(p)-\sum_{p\in \dd kP}\dfrac{(n+1)!-n(p;k)}{(n+1)!}f(p)\right)\\
&+ \left(\dfrac{nk^{-1}}{2\chi(kP)}+\dfrac{O(k^{-2})}{\chi(kP)}\right)\left(\sum_{p\in \dd kP}\dfrac{m(p;k)kf(p)}{(n)!(n+1)}+Vol(kP)\dfrac{f(0)}{(n+1)}\right)\\
=& \dfrac{1}{\chi(kP)}\sum_{p\in kP}f(p) +\dfrac{1}{(n+1)!\chi(kP)} \left(\left(\dfrac{n}{2}+O(k^{-1})\right)m(p;k)- ((n+1)!-n(p;k))\right) \sum_{p\in \dd kP}f(p)
\end{align*}
Therefore, if 
\[ \left(\left(\dfrac{n}{2}+O(k^{-1})\right)m(p;k)- ((n+1)!-n(p;k))\right)\leq 0,\] then the inequality holds. Therefore, if for all $k$, for all $p\in \dd kP$
\[ \dfrac{n}{2}m(p;k)< ((n+1)!-n(p;k)),\] 
then it is asymptotic Chow polystable.
\end{proof}

\begin{rema}
Indeed, by consider $kP$, there is always an integral point $p_0\in kP$. also we can replace $0$ by this point $p_0$. So this criteria is general enough to talk about any $P$. 
\end{rema}
\subsection{$D(X_8)$ and $D(X_9)$}
As $D(X_8)$ and  $D(X_9)$ are not special, we have to triangulate the whole polytopes and compute the inequality directly.\\

Notice that the only way to triangulate $D(X_8)$ and $D(X_9)$ into simplex is  the following: We triangulate $X_8$ and $X_9$ by: 
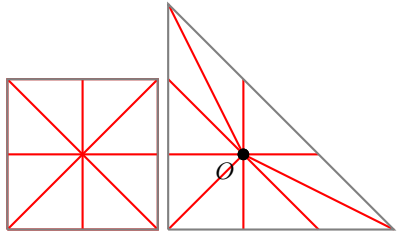
\begin{center}
\begin{tikzpicture}[x=1cm,y=1cm]
\draw[red, thick] (-1,-1) -- (1,-1) -- (1,1)--(-1,1) -- (-1,-1);
\draw[red, thick] (-1,-1) -- (1,1);
\draw[red, thick] (-1,1) -- (1,-1);
\draw[red, thick] (-1,0) -- (1,0);
\draw[red, thick] (0,-1) -- (0,1);


\draw[gray, thick]  (-1,-1) -- (1,-1) -- (1,1)--(-1,1) -- (-1,-1);

\end{tikzpicture}
\begin{tikzpicture}[x=1cm,y=1cm]
\draw[red, thick] (0,-1) -- (0,1);
\draw[red, thick] (-1,0) -- (1,0);
\draw[red,thick] (2,-1) -- (0,0);
\draw[red,thick] (-1,2) -- (0,0);
\draw[red,thick] (-1,-1) -- (0,0);
\draw[red,thick] (-1,1) -- (0,0);
\draw[red,thick] (1,-1) -- (0,0);

\draw[gray, thick] (-1,-1)--(2,-1) --(-1,2) -- (-1,-1);

\foreach \Point/\PointLabel in
{
(0,0)/O,
 }
\draw[fill=black] \Point circle (2pt) node[below left] {$\PointLabel$};
\end{tikzpicture}
\end{center}

Then we connect any small triangle to $(0,0,1)$ and $(0,0,-1)$ to get  3-simplex. Therefore, we can triangulate $D(X_8)$ into 16 simplexes and $D(X_9)$ into 18 simplexes. Then by triangulation of each simplex, we have a triangulation of $kD(X_8)$ and $kD(X_9)$.

As a consequence of Lemma \ref{triangulation of simplex}, we have:
\begin{lemm}\label{triangulation of D}
For $kD(X_i)$, under the above triangulation, 
\[n(p) \left\{ \begin{array}{cl} =i &\text{  if } p=(0,0, \pm k);\\
\leq 24 & \text{  if }  p\in kD(X_i)^o\\
 \leq 12 & \text{ if } p\in \dd kD(X_i)
\end{array}\right.\]
Moreover, the triangulation on $kT_2$ combining with the induced triangulation on $D(X_i)$ onto $\dd kD(X_i)$ gives 
\[m(p) \left\{ \begin{array}{cl} =i &\text{  if } p=(0,0, \pm k);\\
\leq 6 & \text{  if otherwise} 
\end{array}\right.\]
\end{lemm}
As a remark, in here $n_{kP}(0,...,0)=2i$, also for the triangulation of $\dd kD_i$, $n(p)=4$ for $p \in \dd kP$ intersect with the red line. \\

\begin{coro}\label{D(X_i) are Chow stable}
$D(X_8)$ and $D(X_9)$ are asymptotic Chow polystable.
\end{coro}
\begin{proof}
For $p \in \dd kP$ that $p\neq (0,\cdots , \pm 1)$, $n(p;k)\leq \dfrac{(n+1)!}{2}$ and $m(p;k)\leq n!$, then the inequality becomes 
\[n(n!) \leq (n+1)!\] which is true. Also, at $p=(0,...,0, \pm 1)$, we have  \[n(p;k)=m(p;k)=i,\] then we need 
\[(n+1)!> \dfrac{(n+2)}{2}i,\] that is, 
\[1> \dfrac{(n+2)i}{2(n+1)!},\] if $n=3$, then it becomes 
\[1> \dfrac{5i}{48},\] hence this inequality holds for $i<9$. Therefore, by Lemma  \ref{triangulation of D} and theorem \ref{general condition of Chow stable},  $D(X_8)$ and $D(X_9)$ are asymptotic Chow polystable.
\end{proof}

\section{Examples of stability of symmetric reflexive polytopes }
\subsection{1 and 2 dimensional}
\begin{exam}
The only 1 dimensional symmetric reflexive polytopes is $[-1,1]$, which is Chow stable. (See \cite{LLSW19}).
\end{exam}
\begin{exam}
Suppose $P$ is a two dimensional symmetric reflexive polytopes, then it is special, hence it is  asymptotic Chow stable. Indeed, by theorem 1.2 and corollary 3.3 in \cite{LLSW19}, with the fact that $\PP^2$  and $\PP^1\times \PP^1$ (by prop \ref{product of chow stable polytopes}) is Chow stable, so indeed, all 2 dimensional special polytopes are Chow stable. 

As a remark, they are 
\[X_3:= \PP^2/ (\ZZ^/3\ZZ), X_4:= \PP^1\times \PP^1/ (\ZZ/2\ZZ), X_6:=\PP^2   \text{blow up 3 pts}, X_8:=\PP^1\times \PP^1, X_9:=\PP^2,\] and all the line bundle to define the polytopes are $-K_{X_i}$.
\end{exam}

\tikzset{
  font={\fontsize{9pt}{12}\selectfont}}
\begin{center}
\begin{figure}[h!]
\begin{tikzpicture}[x=.5cm,y=.5cm]

\draw[gray, thick] (0,1) -- (-1,-1);

\draw[gray, thick] (-1,-1) -- (1,0);
\draw[gray, thick] (0,1) -- (1,0);
\foreach \Point/\PointLabel in
{
(-1,-1)/,
(0,1)/, (0,0)/, 
 (1,0)/, 
}
\draw[fill=gray] \Point circle (2pt) node[above right] {$\PointLabel$};
\foreach \Point/\PointLabel in
{
(0,0)/,
 }
\draw[fill=black] \Point circle (2pt) node[below left] {$\PointLabel$};

\node at (0,-1.5) {$X_3$};
\end{tikzpicture}
\begin{tikzpicture}[x=.5cm,y=.5cm]

\node  at (0,-1.5) {$X_4$};

\draw[gray, thick] (0,1) -- (-1,0);
\draw[gray, thick] (-1,0) -- (0,-1);
\draw[gray, thick] (1,0) -- (0,-1);
\draw[gray, thick] (1,0) -- (0,1);

\foreach \Point/\PointLabel in
{
(0,-1)/, 
(-1,0)/, (0,0)/, (1,0)/,
 (0,1)/, }
\draw[fill=gray] \Point circle (2pt) node[above right] {$\PointLabel$};
\foreach \Point/\PointLabel in
{
(0,0)/,
 }
\draw[fill=black] \Point circle (2pt) node[below left] {$\PointLabel$};
\end{tikzpicture}
\begin{tikzpicture}[x=.5cm,y=.5cm]

\draw[gray, thick] (1,0) -- (0,1);
\draw[gray, thick] (0,1) -- (-1,1);
\draw[gray, thick] (-1,1) -- (-1,0);
\draw[gray, thick] (-1,0) -- (0,-1);
\draw[gray, thick] (0,-1) -- (1,-1);
\draw[gray, thick] (1,-1) -- ( 1,0);
\node at (0,-1.5) { $X_6$};

\foreach \Point/\PointLabel in
{
(0,-1)/, (1,-1)/,
 (-1,0)/, (0,0)/, (1,0)/,
 (-1,1)/, (0,1)/, 
}
\draw[fill=gray] \Point circle (2pt) node[above right] {$\PointLabel$};
\foreach \Point/\PointLabel in
{
(0,0)/,
 }
\draw[fill=black] \Point circle (2pt) node[below left] {$\PointLabel$};
\end{tikzpicture}
\begin{tikzpicture}[x=.5cm,y=.5cm]

\draw[gray, thick] (-1,-1) -- (2,-1);
\draw[gray, thick] (2,-1) -- (-1,2);
\draw[gray, thick] (-1,2) -- (-1,-1);

\node at (0,-1.5) { $X_8$};

\foreach \Point/\PointLabel in
{
(-1,-1)/,(0,-1)/, (1,-1)/,(2,-1)/,
 (-1,0)/, (0,0)/, (1,0)/,
 (-1,1)/, (0,1)/,
(-1,2)/, 
}
\draw[fill=gray] \Point circle (2pt) node[above right] {$\PointLabel$};
\foreach \Point/\PointLabel in
{
(0,0)/,
 }
\draw[fill=black] \Point circle (2pt) node[below left] {$\PointLabel$};
\end{tikzpicture}
\begin{tikzpicture}[x=.5cm,y=.5cm]

\draw[gray, thick] (-1,-1) -- (1,-1);
\draw[gray, thick] (1,-1) -- (1,1);
\draw[gray, thick] (1,1) -- (-1,1);
\draw[gray, thick] (-1,-1) -- (-1,1);

\node at (0,-1.5) { $X_9$};

\foreach \Point/\PointLabel in
{
(-1,-1)/,(0,-1)/, (1,-1)/,
 (-1,0)/, (0,0)/, (1,0)/,
 (-1,1)/, (0,1)/, (1,1)/,
}
\draw[fill=gray] \Point circle (2pt) node[above right] {$\PointLabel$};
\foreach \Point/\PointLabel in
{
(0,0)/,
 }
\draw[fill=black] \Point circle (2pt) node[below left] {$\PointLabel$};
\end{tikzpicture}
\caption{$X_i$ for $i=3,4,6,8,9$.}\label{2 dim special}
\end{figure}
\end{center}

Notice that in \cite{LLSW19}, there are some examples about the non reflexive polytopes which we will not discuss in detail in this note.
\subsection{3 dimensional}
To study the higher dimensional polytopes, we first recall that given a reflexive polytopes $P$, we can define a reflexive polytopes $\hat{P}$ by the following: Let 
\[\hat{P}:=\{y\in \RR^n\ \langle x, y\rangle \geq -1, \text{ for all }x\in P\}. \] 
If $P$ is symmetric and reflexive, so is $\hat{P}$. 
For example, $D(P)=\widehat{\hat{P}\times \PP^1}$. However, the duality may not share the stability.\\

By Lemma \ref{triangulation of simplex}, if the face are  2 simplex or a rectangle, then for each face, $m(p)=6$ for $p\in P^o$ , $m(p)=3$ if $p\in (\dd P)^o$, and the $m(p)=1$ if $p$ is the vertex. For any  $p\in (\dd P)^o$, there are at most 2 faces connected to $p$, so in order to check if a polytope is special, we only need that there are $i\leq 6$ simplex connecting each vertex under the triangulation on the boundary.\\

Denote $\Delta_0$ be the triangulation of 2 simplex.
As in figure \ref{triangulation of others}, if the face are given by $X_i$, for $i=3,4,6$, then  $n(0)=i$, $n(p)=6$ for $p\in P^o-\{0\}$ , $n(p)=3$ if $p\in (\dd P)^o$, and the $n(p)=2$ if $p$ is the vertex using the rotation of $\Delta_0$ as the triangulation. 
\tikzset{
  font={\fontsize{9pt}{12}\selectfont}}
\begin{center}
\begin{figure}[h!]
\begin{tikzpicture}[x=.5cm,y=.5cm]
\draw[red, thick] (0,2) -- (0,0);
\draw[red, thick] (0,0) -- (2,0);
\draw[red,thick] (0,2) -- (2,0);
\draw[red,thick] (0,0) -- (-2,-2);
\node[text width=1.5cm] at (3.2,1.7) {\red $\triangle_0$};

\draw[gray, thick] (0,2) -- (-2,-2);
\node[text width=1.5cm] at (-2,0) { $\triangle_3$};

\draw[gray, thick] (-2,-2) -- (2,0);

\foreach \Point/\PointLabel in
{
(-2,-2)/,
(-1,0)/,(-1,-1)/,
(0,2)/, (0,1)/, (0,0)/, (0,-1)/,
 (1,1)/, (1,0)/, 
 (2,0)/,}
\draw[fill=gray] \Point circle (2pt) node[above right] {$\PointLabel$};
\foreach \Point/\PointLabel in
{
(0,0)/O,
 }
\draw[fill=black] \Point circle (2pt) node[below left] {$\PointLabel$};
\end{tikzpicture}
\begin{tikzpicture}[x=.5cm,y=.5cm]
\draw[red, thick] (0,2) -- (0,0);
\draw[red, thick] (0,0) -- (2,0);
\draw[red,thick] (0,2) -- (2,0);
\draw[red, thick] (0,0) -- (-2,0);
\draw[red, thick] (0,0) -- (0,-2);
\node[text width=1.5cm] at (3.2,1.7) {\red $\triangle_0$};

\draw[gray, thick] (0,2) -- (-2,0);
\draw[gray, thick] (-2,0) -- (0,-2);
\draw[gray, thick] (2,0) -- (0,-2);
\node[text width=1.5cm] at (-2,0) { $\triangle_4$};

\foreach \Point/\PointLabel in
{
(0,-2)/,
 (-1,-1)/,(0,-1)/, (1,-1)/,
(-2,0)/, (-1,0)/, (0,0)/, (1,0)/,(2,0)/,
 (-1,1)/, (0,1)/, (1,1)/,
(0,2)/,}
\draw[fill=gray] \Point circle (2pt) node[above right] {$\PointLabel$};
\foreach \Point/\PointLabel in
{
(0,0)/O,
 }
\draw[fill=black] \Point circle (2pt) node[below left] {$\PointLabel$};
\end{tikzpicture}
\begin{tikzpicture}[x=.5cm,y=.5cm]
\draw[red, thick] (0,2) -- (0,-2);
\draw[red, thick] (-2,0) -- (2,0);
\draw[red,thick] (0,2) -- (2,0);
\draw[red,thick] (-2,2) -- (2,-2);
\node[text width=1.5cm] at (3.2,1.7) {\red $\triangle_0$};

\draw[gray, thick] (0,2) -- (-2,2);
\draw[gray, thick] (-2,2) -- (-2,0);
\draw[gray, thick] (-2,0) -- (0,-2);
\draw[gray, thick] (0,-2) -- (2,-2);
\draw[gray, thick] (2,-2) -- (2,0);
\draw[gray, thick] (2,0) -- ( 0, 2);
\node[text width=1.5cm] at (-2,0) { $\triangle_6$};

\foreach \Point/\PointLabel in
{
(0,-2)/, (1,-2)/, (2,-2)/, 
 (-1,-1)/,(0,-1)/, (1,-1)/,(2,-1)/, 
(-2,0)/, (-1,0)/, (0,0)/, (1,0)/,(2,0)/,
 (-2,1)/, (-1,1)/, (0,1)/, (1,1)/,
(-2,2)/, (-1,2)/, (0,2)/, 
}
\draw[fill=gray] \Point circle (2pt) node[above right] {$\PointLabel$};
\foreach \Point/\PointLabel in
{
(0,0)/O,
 }
\draw[fill=black] \Point circle (2pt) node[below left] {$\PointLabel$};
\end{tikzpicture}

\caption{$\triangle_0\subset X_3$, $\triangle_0\subset X_4$ and $\tri_0\subset X_6$.}\label{triangulation of others}
\end{figure}
\end{center}

Therefore, if the polytopes which faces are combination of above, then the only possible problem is the vertex, and which can count it one by one.

\begin{prop}\label{all 3 dim symmetric reflexive polytopes}
The following symmetric reflexive 3 dimensional polytopes are asymptotically chow polystable.
\begin{enumerate}
\item $X_i\times [-1,1]$ for i=3,4,6,8,9, which indeed are also special. The corresponding varieties are $X_i\times \PP^1$.
\item  $D(X_i)$, where $i=3,4,6,8,9$. Notice that $X_3$, $X_4$ and $X_6$ are special, but $D(X_8), D(X_9)$ are not, but all of them are asymptotically chow polystable;
\item Other special polytopes, in which they are
\begin{enumerate}
 \item $(\PP^3, O(4)):=Conv\{(-1,-1,-1), (3,-1,-1),(-1,3,-1), (-1,-1,3)\}$ ( tetrahedron ) and its dual, \\
  $A_3=Conv\{(-1,-1,-1), (1,0,0), (0,1,0),(0,0,1)\}$   ( tetrahedron );
\item  $\PP^3$ blowup 4 points, which is a convex set of the points:\\
  $(0,-1,-1),(-1,0,-1), (-1,-1,0). (2,-1,-1),(2,-1,0),(2,0,-1)$,\\
$ (-1,2,-1),(-1,2,0), (0,2,-1), (-1,-1,2), (-1,0,2), (0,-1,2)$,\\
 hence it is a truncated tetrahedron, which the boundary contain 4 $X_6$ and 4 simplex triangle. each vertex connect to 2 simplex triangle and 2 $X_6$.\\
It dual is given by 
\[Conv\{(\pm 1, 0, 0) (0,\pm 1, 0), (0, 0, .\pm 1), (-1,-1,-1),(1,1,1)\},\]
which is $D(X_4)$ glue with two standard 3 simplex, which the faces are all standard 2 simplex.
\item $Conv\{ (\pm 1,0, 0),(0, \pm 1, 0), (\pm 1, \mp 1, 0), (0, 0, \pm 1), (\pm 1, 0, \mp 1), (0, \pm 1, \mp 1)\}$, which is a cuboctahedron, with 8 triangular faces and 6 square faces, and each vertex are connected to 2 2 simplex and 2 square. The action group is $S_3$ which permute the coordinates. Hence for the surface, each vertex $p$, 
\[n(p)\leq 2+2 \cdot 2=6,\] and the remaining is also smaller than 6.\\
and its dual, which is given by the convex hull of the points:
\[(1,0,0),(1,1,0),(0,1,0),(-1,0,0),(-1,-1,0), (0,-1,0), \] \[(0,0,1),(1,0,1),(1,1,1), (0,1,1), (0,0,-1),(-1,0,-1),(-1,-1,-1), (0,-1,-1)\] which is a rhombic dodecahedron. Notice that $(0,1,1), (1,0,1), (0,-1,-1),(-1,0,-1)$ has 4 square touched the points, and other are only $3$. So when we triangulate the surface, if we choose the triangulation such that two of the square don't bisect along those points, then for any point $p$ in it, 
\[n(p)\leq 6,\] hence it is special. 

\end{enumerate}
\end{enumerate}
\end{prop}

\subsection{higher dimensional}
We know that in high dimensional, not every symmetric reflexive polytopes are asymptotic Chow stable, for example $ D([-1,1]^n)$ for $n\geq 6$. On the other hand, we would provide two classes of polytope which are special

\begin{exam}\label{A_n}
Consider $A_n:=\{[z_0,...,z_{n+1}]\in \PP^{n+1}|z_1...z_{n+1}=z_{0}^{n+1}\}$. The corresponding polytopes, also denoted as $A_n$, are given by
\[A_n=Conv\{(1,...,0),(0,1,...,0), ..., (0,...,0,1), (-1,-1,...,-1)\}.\] As all the face are simplexes, with all the condimension 2 or above boundry intersect with less than $n$, it is asymptotic  Chow polystable.  

Notice that $A_2= X_3$ in our notation on  symmetric reflexive polygons. Also, as a polytope, $A_n$ is the dual polytope of the polytope corresponding to $(\PP^n, O(n+1))$.
\end{exam}
\begin{proof} 
Notice that the surface is given by $n$ piece of $\PP^{n-1}$. So we only need to know how many simplex will attend to the point in the co-dimension k skeleton.\\

 Let $a_i=e_i$ and $a_{n-1}= (-1,...,-1)$, then we can represent any codimensional k piece by the following: $a_1$ represent the point $e_1$, $\{a_1,a_2\}$represent the segment containing $a_1,a_2$, and etc. Also, as symmetric group $S_{n+1}$ act on it, we only need to consider how many face containing the n-r skeleton $\{a_1...a_{n-r}\}$. But the face containing $a_1... a_{n-r}$ is represent by the set $\{a_{n-r+1}... a_{n+1}\}$ removing one element. therefore, there is $r+1$ face connecting the skeleton containing $a_1...a_{n-r}$. Therefore,    lemma \ref{triangulation of simplex} implies that for any point in $p\in \dd kP$ which is in the interior of $n-r$ skeleton, \[n(p)=\dfrac{(n)!}{(r)!}(r)=\dfrac{n!}{(r-1)!}\leq n!.\] Therefore, it has regular triangulation. Also, it is symmetric and reflexive, hence it is special, which implies it is asymptotic Chow polystable.  
\end{proof}

\begin{exam}
Consider $(\PP^n, O(n+1))$. The boundary of $k(\PP^n, O(n+1))$ is defined by \[\bigcap_{i=1}^n\{x_i=-k\}\cap \{x_1+\cdots x_n=k\}.\] Up to an $S_{n+1}$ action, a point $p$ is in the interior of codimensional r skeleton if 
\[p=(-k,...,-k,v),\] where $v\in \RR^{n-r}$ such that \[-r+v_1+\cdots +v_{n-r}<k,\] \[v_i>-k\] for all $i=1,...,n-k$. Hence $n(p)=r$. So we have the same calculation of $A_n$, which implies it is special, and therefore it is asymptotic Chow polystable.  
\end{exam}
\begin{exam}\label{D_n}
Define $D_n:=Conv\{(\pm 1,...,0), (0, \pm 1, 0,...,0), ..., (0,...,0, \pm 1) \}$. $D_n$ is special for all $n$. Notice that $D_2=X_4$ and $D_3= D(X_4)$.
\end{exam}
\begin{proof}
Notice that $k D_n=\ZZ_2^n \cdot kT_n$, where $\ZZ_2^n=\{1,-1\}^n$ with the group action to be multiplication, and the action is multiplication to the corresponding coordinate. $D_n$ is symmetric and reflexive. To show $D_n$ has regular boundary, $p$ is in the interior of codimension $r+1$ skeleton if  $p=(x_1,...,x_n) \in \dd  k D_n$ with 
\[x_{i_1}=\cdots =x_{i_r}=0\] for $r=0,...,n-1$. we denotes these points as $p_r$.  Hence, similar to the last example, as a consequence of lemma \ref{triangulation of simplex}, we have
\[n(p_r)=\dfrac{(n)!}{(r+1)!}(2^r)=\left(\dfrac{2}{r+1}\right) \cdots \left(\dfrac{2}{2}\right) n! \leq n!,\] hence it is special.
\end{proof}
\begin{exam}
Notice that $D_6$ is the dual polytope of $[-1,1]^6$, and thus $D_6\times [-1,1]$ (i.e., $D_6\times \PP^1$ as the corresponding variety) is  asymptotic Chow polystable.  However, its dual is $D([-1,1]^6)$, therefore a dual of a asymptotic Chow polystable polytope need not to be asymptotic Chow polystable (or even semistable).
\end{exam}


\section{Pictures of 3 dimensional symmetric reflexive polytopes}
Finally, we provides the picture of the 3 dimensional symmetric reflexive polytopes mentioned above. \\

\tikzset{
  font={\fontsize{9pt}{12}\selectfont}}
\begin{center}
\begin{figure}[h!]
\tdplotsetmaincoords{70}{110}
\begin{tikzpicture}[tdplot_main_coords]

\draw[gray, thick] (-1,-1,1) -- (1,0,1) -- (0,1,1) -- (-1,-1,1);
\draw[gray, thick] (-1,-1,-1) -- (1,0,-1) -- (0,1,-1) -- (-1,-1,-1);
\draw[gray, thick] (-1,-1,1) -- (-1,-1,-1) ;
\draw[gray, thick]  (0,1,-1) -- (0,1,1) ;
\draw[gray, thick] (1,0,1) -- (1,0,-1);

\foreach \Point/\PointLabel in
{
 (-1,-1,1)/, (1,0,1)/, (0,1,1)/,
(-1,-1,0)/, (1,0,0)/, (0,1,0)/,
(-1,-1,-1)/, (1,0,-1)/, (0,1,-1)/,
(0,0,1)/,  (0,0,-1)/,
}
\draw[fill=gray] \Point circle (2pt) node[above right] {$\PointLabel$};

\foreach \Point/\PointLabel in
{
 (0,0,0)/, 
 }
\draw[fill=green] \Point circle (2pt) node[below left] {$\PointLabel$};
\node at (0,0,-1.7) {$X_3\times [-1,1]$};
\end{tikzpicture}
\begin{tikzpicture}[tdplot_main_coords]
\draw[gray, thick] (0,-1,1) -- (1,0,1) -- (0,1,1) -- (-1,0,1) -- (0,-1,1);
\draw[gray, thick] (0,-1,-1) -- (1,0,-1) -- (0,1,-1) -- (-1,0,-1) -- (0,-1,-1);
\draw[gray, thick] (-1,0,1) -- (-1,0,-1) ;
\draw[gray, thick]  (0,1,-1) -- (0,1,1) ;
\draw[gray, thick]  (0,-1,-1) -- (0,-1,1) ;
\draw[gray, thick] (1,0,1) -- (1,0,-1);

\foreach \Point/\PointLabel in
{
(0,-1, 1)/,(-1,0,1)/, (1,0,1)/, (0,1,1)/, 
(0,-1, 0)/,(-1,0,0)/, (1,0,0)/, (0,1,0)/, 
(0,-1, -1)/,(-1,0,-1)/, (1,0,-1)/, (0,1,-1)/,
(0,0,1)/,  (0,0,-1)/, 
}
\draw[fill=gray] \Point circle (2pt) node[above right] {$\PointLabel$};

\foreach \Point/\PointLabel in
{
 (0,0,0)/, 
 }
\draw[fill=green] \Point circle (2pt) node[below left] {$\PointLabel$};
\node at (0,0,-1.7) {$X_4\times [-1,1]$};
\end{tikzpicture}
\begin{tikzpicture}[tdplot_main_coords]
\draw[gray, thick] (0,-1,1) -- (1,-1,1) -- (1,0,1) -- (0,1,1) --(-1,1,1) --  (-1,0,1) -- (0,-1,1);
\draw[gray, thick] (0,-1,-1)  -- (1,-1,-1)-- (1,0,-1) -- (0,1,-1)  --(-1,1,-1) -- (-1,0,-1) -- (0,-1,-1);
\draw[gray, thick] (-1,0,1) -- (-1,0,-1) ;
\draw[gray, thick]  (0,1,-1) -- (0,1,1) ;
\draw[gray, thick]  (0,-1,-1) -- (0,-1,1) ;
\draw[gray, thick] (1,0,1) -- (1,0,-1);
\draw[gray, thick] (-1,1,1) -- (-1,1,-1) ;
\draw[gray, thick] (1,-1,1) -- (1,-1,-1) ;

\foreach \Point/\PointLabel in
{
(0,-1, 1)/,(-1,0,1)/, (1,0,1)/, (0,1,1)/, (1,-1,1)/, (-1,1,1)/,
(0,-1, 0)/,(-1,0,0)/, (1,0,0)/, (0,1,0)/, (1,-1,0)/, (-1,1,0)/,
(0,-1, -1)/,(-1,0,-1)/, (1,0,-1)/, (0,1,-1)/, (1,-1,-1)/, (-1,1,-1)/,
(0,0,1)/,  (0,0,-1)/,
}
\draw[fill=gray] \Point circle (2pt) node[above right] {$\PointLabel$};

\foreach \Point/\PointLabel in
{
 (0,0,0)/, 
 }
\draw[fill=green] \Point circle (2pt) node[below left] {$\PointLabel$};
\node at  (0,0,-1.7)  {$X_6\times [-1,1]$};
\end{tikzpicture}

\begin{tikzpicture}[tdplot_main_coords]
\draw[gray, thick] (-1,-1,1) -- (1,-1,1) -- (1,1,1) -- (-1,1,1) -- (-1,-1,1);
\draw[gray, thick] (-1,-1,-1) -- (1,-1,-1) -- (1,1,-1) -- (-1,1,-1) -- (-1,-1,-1);
\draw[gray, thick] (-1,-1,1) -- (-1,-1,-1) ;
\draw[gray, thick]  (-1,1,-1) -- (-1,1,1) ;
\draw[gray, thick]  (1,-1,-1) -- (1,-1,1) ;
\draw[gray, thick] (1,1,1) -- (1,1,-1);

\foreach \Point/\PointLabel in
{
(0,-1, 1)/,(-1,0,1)/, (1,0,1)/, (0,1,1)/, (-1,-1,1)/, (1,-1,1)/, (1,1,1)/, (-1,1,1)/,
(0,-1, 0)/,(-1,0,0)/, (1,0,0)/, (0,1,0)/, (-1,-1,0)/, (1,-1,0)/, (1,1,0)/, (-1,1,0)/,
(0,-1, -1)/,(-1,0,-1)/, (1,0,-1)/, (0,1,-1)/, (-1,-1,-1)/, (1,-1,-1)/, (1,1,-1)/, (-1,1,-1)/,
(0,0,1)/,  (0,0,-1)/, 
}
\draw[fill=gray] \Point circle (2pt) node[above right] {$\PointLabel$};

\foreach \Point/\PointLabel in
{
 (0,0,0)/, 
 }
\draw[fill=green] \Point circle (2pt) node[below left] {$\PointLabel$};
\node at (0,0,-1.7) {$X_8\times [-1,1]$};
\end{tikzpicture}
\begin{tikzpicture}[tdplot_main_coords]

\draw[gray, thick] (-1,-1,1) -- (-1,2,1) -- (2,-1,1) -- (-1,-1,1);
\draw[gray, thick] (-1,-1,-1) -- (-1,2,-1) -- (2,-1,-1) -- (-1,-1,-1);
\draw[gray, thick] (-1,-1,1) -- (-1,-1,-1) ;
\draw[gray, thick]  (2,-1,-1) -- (2,-1,1) ;
\draw[gray, thick] (-1,2,1) -- (-1, 2,-1);

\foreach \Point/\PointLabel in
{
 (-1,-1,1)/, (1,0,1)/, (0,1,1)/, (-1,2,1)/, (2,-1,1)/, (-1,0,1)/,(-1,1,1)/, (1,-1,1)/, (0,-1,1)/,
(-1,-1,0)/, (1,0,0)/, (0,1,0)/,  (-1,2,0)/, (2,-1,0)/, (-1,0,0)/,(-1,1,0)/, (1,-1,0)/, (0,-1,0)/,
(-1,-1,-1)/, (1,0,-1)/, (0,1,-1)/,  (-1,2,-1)/, (2,-1,-1)/, (-1,0,-1)/,(-1,1,-1)/, (1,-1,-1)/, (0,-1,-1)/,
(0,0,1)/,  (0,0,-1)/,
}
\draw[fill=gray] \Point circle (2pt) node[above right] {$\PointLabel$};

\foreach \Point/\PointLabel in
{
 (0,0,0)/, 
 }
\draw[fill=green] \Point circle (2pt) node[below left] {$\PointLabel$};
\node at (0,0,-1.7) {$X_9\times [-1,1]$};
\end{tikzpicture}
\caption{$X_i \times [-1,1]$ }\label{triangulation of rectangle}
\end{figure}
\end{center}

\tikzset{
  font={\fontsize{9pt}{12}\selectfont}}
\begin{center}
\begin{figure}[h!]
\tdplotsetmaincoords{70}{110}
\begin{tikzpicture}[tdplot_main_coords]

\draw[gray, thick] (-1,-1,0) -- (1,0,0) -- (0,1,0) -- (-1,-1,0);
\draw[gray, thick] (-1,-1,0) -- (0,0,1) ;
\draw[gray, thick]  (0,1,0) -- (0,0,1) ;
\draw[gray, thick] (1,0,0) -- (0,0,1);
\draw[gray, thick] (1,0,0) -- (0,0,-1);
\draw[gray, thick] (-1,-1,0) -- (0,0,-1) ;
\draw[gray, thick]  (0,1,0) -- (0,0,-1) ;

\foreach \Point/\PointLabel in
{
(-1,-1,0)/, (1,0,0)/, (0,1,0)/,
(0,0,1)/,  (0,0,-1)/,
}
\draw[fill=gray] \Point circle (2pt) node[above right] {$\PointLabel$};

\foreach \Point/\PointLabel in
{
 (0,0,0)/, 
 }
\draw[fill=green] \Point circle (2pt) node[below left] {$\PointLabel$};
\node at (0,0,-1.7) {$D(X_3)$};
\end{tikzpicture}
\begin{tikzpicture}[tdplot_main_coords]
\draw[gray, thick] (0,-1,0) -- (1,0,0) -- (0,1,0) -- (-1,0,0) -- (0,-1,0);
\draw[gray, thick] (-1,0,0) -- (0,0,1) ;
\draw[gray, thick]  (0,1,0) -- (0,0,1) ;
\draw[gray, thick]  (0,-1,0) -- (0,0,1) ;
\draw[gray, thick] (1,0,0) -- (0,0,1);
\draw[gray, thick] (-1,0,0) -- (0,0,-1) ;
\draw[gray, thick]  (0,1,0) -- (0,0,-1) ;
\draw[gray, thick]  (0,-1,0) -- (0,0,-1) ;
\draw[gray, thick] (1,0,0) -- (0,0,-1);

\foreach \Point/\PointLabel in
{
(0,-1, 0)/,(-1,0,0)/, (1,0,0)/, (0,1,0)/, 
(0,0,1)/,  (0,0,-1)/, 
}
\draw[fill=gray] \Point circle (2pt) node[above right] {$\PointLabel$};

\foreach \Point/\PointLabel in
{
 (0,0,0)/, 
 }
\draw[fill=green] \Point circle (2pt) node[below left] {$\PointLabel$};
\node at (0,0,-1.7) {$D(X_4)$};
\end{tikzpicture}
\begin{tikzpicture}[tdplot_main_coords]
\draw[gray, thick] (0,-1,0) -- (1,-1,0) -- (1,0,0) -- (0,1,0) --(-1,1,0) --  (-1,0,0) -- (0,-1,0);
\draw[gray, thick] (-1,0,0) -- (0,0,1) ;
\draw[gray, thick]  (0,1,0) --(0,0,1) ;
\draw[gray, thick]  (0,-1,0) -- (0,0,1) ;
\draw[gray, thick] (1,0,0) -- (0,0,1) ;
\draw[gray, thick] (-1,1,0) -- (0,0,1) ;
\draw[gray, thick] (1,-1,0) -- (0,0,1) ;

\draw[gray, thick] (-1,0,0) -- (0,0,-1) ;
\draw[gray, thick]  (0,1,0) --(0,0,-1) ;
\draw[gray, thick]  (0,-1,0) -- (0,0,-1) ;
\draw[gray, thick] (1,0,0) -- (0,0,-1) ;
\draw[gray, thick] (-1,1,0) -- (0,0,-1) ;
\draw[gray, thick] (1,-1,0) -- (0,0,-1) ;

\foreach \Point/\PointLabel in
{
(0,-1, 0)/,(-1,0,0)/, (1,0,0)/, (0,1,0)/, (1,-1,0)/, (-1,1,0)/,
(0,0,1)/,  (0,0,-1)/,
}
\draw[fill=gray] \Point circle (2pt) node[above right] {$\PointLabel$};

\foreach \Point/\PointLabel in
{
 (0,0,0)/, 
 }
\draw[fill=green] \Point circle (2pt) node[below left] {$\PointLabel$};
\node at  (0,0,-1.7)  {$D(X_6)$};
\end{tikzpicture}

\begin{tikzpicture}[tdplot_main_coords]
\draw[gray, thick]  (1,1,0) -- (-1,1,0);
\draw[gray, dashed] (-1,-1,0) -- (1, -1, 0) ;
\draw[gray, dashed] (-1,1,0) -- (-1,-1,0);
\draw[gray, thick]  (1,-1,0) --  (1,1,0) ;
\draw[gray, thick] (-1,-1,0) -- (0,0,1) ;
\draw[gray, thick]  (-1,1,0) --  (0,0,1);
\draw[gray, thick]  (1,-1,0) --  (0,0,1) ;
\draw[red, thick] (1,1,0) --  (0,0,1);

\draw[gray, dashed] (-1,-1,0) -- (0,0,-1) ;
\draw[gray,  dashed]  (-1,1,0) --  (0,0,-1);
\draw[gray, thick]  (1,-1,0) --  (0,0,-1) ;
\draw[gray, thick] (1,1,0) --  (0,0,-1);

\draw[red, thick] (-1,1,0) --  (0,0,1);
\draw[red, thick] (-1,1,0) -- (1,1,0);
\draw[red, thick] (0,1,0) --  (0,0,1);

\foreach \Point/\PointLabel in
{
(0,-1, 0)/,(-1,0,0)/, (1,0,0)/, (0,1,0)/, (-1,-1,0)/, (1,-1,0)/, (1,1,0)/, (-1,1,0)/,
(0,0,1)/,  (0,0,-1)/, 
}
\draw[fill=gray] \Point circle (2pt) node[above right] {$\PointLabel$};

\foreach \Point/\PointLabel in
{
 (0,0,0)/, 
 }
\draw[fill=green] \Point circle (2pt) node[below left] {$\PointLabel$};
\node at (0,0,-1.7) {$D(X_8)$};
\end{tikzpicture}
\begin{tikzpicture}[tdplot_main_coords]

\draw[gray, dashed] (-1,-1,0) -- (-1,2,0) ;
\draw[red,thick]   (-1,2,0) -- (2,-1,0);
\draw[gray,dashed]  (2,-1,0) -- (-1,-1,0);

\draw[gray, thick] (-1,-1,0) -- (0,0,1) ;
\draw[red, thick]  (2,-1,0) -- (0,0,1) ;
\draw[red, thick] (-1,2,0) -- (0,0,1);

\draw[red, thick]  (0,1,0) -- (0,0,1) ;
\draw[red, thick]  (1,0,0) -- (0,0,1) ;
\draw[gray,dashed] (-1,-1,0) -- (0,0,-1) ;
\draw[gray, thick]  (2,-1,0) -- (0,0,-1) ;
\draw[gray, thick] (-1,2,0) -- (0,0,-1);

\foreach \Point/\PointLabel in
{
(-1,-1,0)/, (1,0,0)/, (0,1,0)/,  (-1,2,0)/, (2,-1,0)/, (-1,0,0)/,(-1,1,0)/, (1,-1,0)/, (0,-1,0)/,
(0,0,1)/,  (0,0,-1)/,
}
\draw[fill=gray] \Point circle (2pt) node[above right] {$\PointLabel$};

\foreach \Point/\PointLabel in
{
 (0,0,0)/, 
 }
\draw[fill=green] \Point circle (2pt) node[below left] {$\PointLabel$};
\node at (0,0,-1.7) {$D(X_9)$};
\end{tikzpicture}
\caption{$D(X_i)$ }\label{triangulation of double cone}
\end{figure}
\end{center}

\tikzset{
  font={\fontsize{9pt}{12}\selectfont}}
\begin{center}
\begin{figure}[h!]
\tdplotsetmaincoords{70}{110}
\begin{tikzpicture}[tdplot_main_coords]

\draw[gray, thick]  (3,-1,-1)--(-1,3,-1) -- (-1,-1,3) -- (3,-1,-1);
\draw[gray, dashed]  (3,-1,-1)--(-1,-1,-1) -- (-1,-1,3) -- (3,-1,-1);
\draw[gray, dashed]  (3,-1,-1)--(-1,3,-1) -- (-1,-1,-1) -- (3,-1,-1);
\foreach \Point/\PointLabel in
{
(-1,-1,-1)/, (0,-1,-1)/, (1,-1,-1)/, (2,-1,-1)/, (3,-1,-1)/, 
(-1,0,-1)/,  (0,0,-1)/, (1,0,-1)/, (2,0,-1)/, 
(-1,1,-1)/,  (0,1,-1)/, (1,1,-1)/,  
(-1,2,-1)/,  (0,2,-1)/, 
(-1,3,-1)/,
(-1,-1,3)/,
}
\draw[gray] \Point circle (2pt) node[above right] {$\PointLabel$};

\foreach \Point/\PointLabel in
{
(-1,-1,0)/, (0,-1,0)/, (1,-1,0)/, (2,-1,0)/, 
(-1,0,0)/,  (0,0,0)/, (1,0,0)/, 
(-1,1,0)/,  (0,1,0)/, 
(-1,2,0)/, 
(-1,-1,1)/, (0,-1,1)/, (1,-1,1)/,  
(-1,0,1)/,  (0,0,1)/,
(-1,1,1)/, 
(-1,-1,2)/, (0,-1,2)/,   
(-1,0,2)/,  
}
\draw[gray] \Point circle (2pt) node[above right] {$\PointLabel$};
\foreach \Point/\PointLabel in
{
(-1,-1,-1)/,  (3,-1,-1)/, 
(-1,-1,3)/, (-1,3,-1)/,
}
\draw[fill=gray] \Point circle (2pt) node[above right] {$\PointLabel$};

\foreach \Point/\PointLabel in
{
 (0,0,0)/, 
 }
\draw[fill=green] \Point circle (2pt) node[below left] {$\PointLabel$};

\node at (0,0,-2) {$\PP^3$};
\end{tikzpicture}
\begin{tikzpicture}[tdplot_main_coords]

\draw[gray, thick]  (1,0,0)--(0,1,0) -- (0,0,1) -- (1,0,0);
\draw[gray,thick]  (1,0,0)--(-1,-1,-1) -- (0,0,1) -- (1,0,0);
\draw[gray, thick] (1,0,0)--(0,1,0) -- (-1,-1,-1)-- (1,0,0);

\foreach \Point/\PointLabel in
{
(1,0,0)/,  (0,1,0)/, 
(0,0,1)/, (-1,-1,-1)/,
}
\draw[fill=gray] \Point circle (2pt) node[above right] {$\PointLabel$};

\foreach \Point/\PointLabel in
{
 (0,0,0)/, 
 }
\draw[fill=green] \Point circle (2pt) node[below left] {$\PointLabel$};
\node at (0,0,-1.3) {$\PP^3/(\ZZ/4\ZZ)$};
\end{tikzpicture}
\begin{tikzpicture}

\foreach \Point/\PointLabel in
{
 (0,-1,-1)/,(-1,0,-1)/, (-1,-1,0)/, (2,-1,-1)/,(2,-1,0)/,(2,0,-1)/,
 (-1,2,-1)/,(-1,2,0)/, (0,2,-1)/, (-1,-1,2)/, (-1,0,2)/, (0,-1,2)/,
}
\draw[fill=gray] \Point circle (2pt) node[above right] {$\PointLabel$};

\foreach \Point/\PointLabel in
{
(0,-1,-1)/, (1,-1,-1)/, (2,-1,-1)/, 
(-1,0,-1)/,  (0,0,-1)/, (1,0,-1)/, (2,0,-1)/, 
(-1,1,-1)/,  (0,1,-1)/, (1,1,-1)/,  
(-1,2,-1)/,  (0,2,-1)/, 
}
\draw[gray] \Point circle (2pt) node[above right] {$\PointLabel$};

\foreach \Point/\PointLabel in
{
(-1,-1,0)/, (0,-1,0)/, (1,-1,0)/, (2,-1,0)/, 
(-1,0,0)/,  (0,0,0)/, (1,0,0)/, 
(-1,1,0)/,  (0,1,0)/, 
(-1,2,0)/, 
(-1,-1,1)/, (0,-1,1)/, (1,-1,1)/,  
(-1,0,1)/,  (0,0,1)/,
(-1,1,1)/, 
(-1,-1,2)/, (0,-1,2)/,   
(-1,0,2)/,  
}
\draw[gray] \Point circle (2pt) node[above right] {$\PointLabel$};
\foreach \Point/\PointLabel in
{
 (0,0,0)/, 
 }
\draw[fill=green] \Point circle (2pt) node[below left] {$\PointLabel$};

\draw[gray, thick]  (2,-1,-1)--(2,-1,0) -- (2,0,-1) -- (2,-1,-1);
\draw[gray, thick]  (-1,2,-1)--(-1,2,0) -- (0,2,-1) -- (-1,2,-1);
\draw[gray, dashed]  (-1,0,-1)--(-1,-1,0) -- (0,-1,-1) -- (-1,0,-1);
\draw[gray, thick]  (-1,-1,2)--(-1,0,2) -- (0,-1,2) -- (-1,-1,2);
\draw[gray, thick]  (2,-1,0)--(0,-1,2); 
\draw[gray, thick]  (0,2,-1)--(2,0,-1); 
\draw[gray, thick]  (-1,0,2)--(-1,2,0); 
\draw[gray, dashed]  (-1,0,-1)--(-1,2,-1);
\draw[gray, dashed]  (-1,-1,0)--(-1,-1,2);
\draw[gray, dashed]  (0,-1,-1)--(2,-1,-1);

\node at (0,-2,0) {$\PP^3$ blows up 4 points};

\end{tikzpicture}
\tdplotsetmaincoords{10}{10}
\begin{tikzpicture}[tdplot_main_coords]

\draw[blue, thick]  (1,0,0)--(1,1,1);
\draw[blue, thick]  (0,1,0)--(1,1,1);
\draw[blue, thick]  (0,0,1)--(1,1,1);
\draw[blue,thick]  (-1,0,0)--(-1,-1,-1);
\draw[blue,thick]  (0,-1,0)--(-1,-1,-1);
\draw[blue, thick]  (0,0,-1)--(-1,-1,-1);

\draw[gray,thick] (0,-1,0) -- (1,0,0) -- (0,1,0) -- (-1,0,0) -- (0,-1,0);
\draw[gray, thick]  (1,0,0) -- (0,1,0);
\draw[gray, thick]  (1,0,0) -- (0,-1,0);
\draw[gray, thick] (-1,0,0) -- (0,0,1) ;
\draw[gray,thick]  (0,1,0) -- (0,0,1) ;
\draw[gray, thick]  (0,-1,0) -- (0,0,1) ;
\draw[gray, thick] (1,0,0) -- (0,0,1);
\draw[gray, thick] (-1,0,0) -- (0,0,-1) ;
\draw[gray, thick]  (0,1,0) -- (0,0,-1) ;
\draw[gray, thick]  (0,-1,0) -- (0,0,-1) ;
\draw[gray, thick] (1,0,0) -- (0,0,-1);

\foreach \Point/\PointLabel in
{
(1, 0, 0)/, (-1, 0, 0)/, (0,1, 0)/, (0,-1, 0)/, (0, 0, 1)/,  (0, 0, -1)/, (-1,-1,-1)/, (1,1,1)/,
}
\draw[fill=gray] \Point circle (2pt) node[above right] {$\PointLabel$};

\foreach \Point/\PointLabel in
{
 (0,0,0)/, 
 }
\draw[fill=green] \Point circle (2pt) node[below left] {$\PointLabel$};

\node at (0,-1,-1.3) {Dual of $\PP^3$ blow up 4 points. };

\end{tikzpicture}
\tdplotsetmaincoords{70}{110}
\begin{tikzpicture}[tdplot_main_coords]

\draw[gray, thick] (-1,0,1) -- (0,-1,1) -- (0,0,1) -- (-1,0,1);
\draw[gray, thick] (1,0,-1) -- (0,1,-1) -- (0,0,-1) -- (1,0,-1);
\draw[gray, dashed] (1,0,0)-- (0,1,0) --  (-1,1,0) -- (-1,0,0) -- (0,-1,0) -- (1,-1,0) -- (1,0,0);
\draw[gray, thick]  (1,-1,0) -- (1,0,0) -- (0,1,0) --(-1,1,0);
\draw[gray, thick]   (1,0,0) -- (0,0,1);
\draw[gray, thick]   (0,1,0) -- (0,0,1);
\draw[gray, thick]   (1,-1,0) -- (0,-1,1);
\draw[gray, thick]   (-1,1,0) -- (-1,0,1);
\draw[gray, dashed]   (0,-1,0) -- (0,-1,1);
\draw[gray, dashed]   (-1,0,0) -- (-1,0,1);
\draw[gray,thick]   (0,1,0) -- (0,1,-1);
\draw[gray, thick]   (1,0,0) -- (1,0,-1);
\draw[gray, dashed]   (-1,0,0) -- (0,0,-1);
\draw[gray, dashed]   (0,-1,0) -- (0,0,-1);
\draw[gray, thick]   (-1,1,0) -- (0,1,-1);
\draw[gray, thick]   (1,-1,0) -- (1,0,-1);
\foreach \Point/\PointLabel in
{
(-1,0,1)/,(0,-1,1)/, (0,0,1)/,
(1,0,0)/, (0,1,0)/, (-1,1,0)/,(-1,0, 0)/, (0,-1,0)/, (1,-1,0)/,
(1,0,-1)/, (0,1,-1)/, (0,0,-1)/,
}
\draw[fill=gray] \Point circle (2pt) node[above right] {$\PointLabel$};

\foreach \Point/\PointLabel in
{
 (0,0,0)/, 
 }
\draw[fill=green] \Point circle (2pt) node[below left] {$\PointLabel$};

\node at (0.5,-1,-1.5) {cuboctahedron};

\end{tikzpicture}
\tdplotsetmaincoords{30}{60}
\begin{tikzpicture}[tdplot_main_coords]

\draw[gray, thick] (0,0,1)--(1,0,1) --(1,1,1)-- (0,1,1)-- (0,0,1);
\draw[gray, thick] (0,0,-1)--(-1,0,-1) --(-1,-1,-1)-- (0,-1,-1)-- (0,0,-1);
\draw[gray, thick] (1,0,1) --(1,1,1)-- (0,1,1);
\draw[gray, thick] (1,0,0) -- (1,0,1);
\draw[gray, thick] (1,1,0) -- (1,1,1);
\draw[gray, thick] (0,1,0) -- (0,1,1);
\draw[gray, thick] (-1,-1,0) -- (0,0,1);
\draw[gray, thick]  (-1,0,0)--(0,1,1);
\draw[gray, thick] (0,-1,0) -- (1,0,1);

\draw[gray, thick] (1,0,0) --(1,1,0)-- (0,1,0);
\draw[gray, thick] (-1,0,0) --(-1,-1,0)-- (0,-1,0);

\draw[red, dashed] (1,1,1) --(0,0, 1);
\draw[red, dashed] (-1,-1,-1) --(0,0, -1);
\draw[red, dashed] (0,1,0) --(-1,0, 0);
\draw[red, dashed] (0,-1,0) --(1,0, 0);
\draw[gray,thick] (1,0,0) -- (0,-1,-1);
\draw[gray, thick] (1,1,0) -- (0,0,-1);
\draw[gray, thick] (0,1,0) -- (-1,0,-1);
\draw[gray, thick] (0,-1,0) -- (0,-1,-1);
\draw[gray, thick] (-1,-1,0) -- (-1,-1,-1);
\draw[gray, thick] (-1,0,0) -- (-1,0,-1);
\foreach \Point/\PointLabel in
{
(1,0,0)/,(1,1,0)/,(0,1,0)/, (-1,0,0)/, (-1,-1,0)/, (0,-1,0)/,
(0,0,1)/,(1,1,1)/,  (0,0,-1)/,(-1,-1,-1)/, 
}
\draw[fill=gray] \Point circle (2pt) node[above right] {$\PointLabel$};

\foreach \Point/\PointLabel in
{
(1,0,1)/, (0,1,1)/, (-1,0,-1)/,(0,-1,-1)/,
}
\draw[fill=blue] \Point circle (2pt) node[above right] {$\PointLabel$};

\foreach \Point/\PointLabel in
{
 (0,0,0)/, 
 }
\draw[fill=green] \Point circle (2pt) node[below left] {$\PointLabel$};

\node at (0,-1,-1.5) {Daul of cuboctahedron. };

\end{tikzpicture}

\caption{Others special polytopes, red line indicate part of triangulation}
\end{figure}
\end{center}


\newpage
\appendix
 \section{From integral polytypes to varieties}
In this section, we will briefly explain how we obtain a toric variety from the integral polytope. then we will write down the corresponding varieties of the toric varieties occurred in this note.
\subsection{general procedure}
Let $P$ be a integral polytope containing $(0,...,0)$. Let $\{p_0=(0,...,0), p_1,...,p_{N}\}$ be all the integral points in $P$. Then we can define a toric subvariety in $\PP^N$ by the following equations:

Suppose we have \[c_1p_{i_1}+\cdots+ c_rp_{i_r}=b_1 p_{j_1}+\cdots +b_s p_{j_s},\] and without loss of generality, we may assume \[c_1+\cdots +c_r=b_1+\cdots +b_s+a\] for some $a \geq 0$. Then we have a homogeneous polynomial defined by 
\[z_{i_1}^{c_1}\cdots z_{i_r}^{c_r}= z_{j_1}^{b_1}\cdots z_{j_s}^{b_s}z_0^a.\]  Then 
$\{z_{i_1}^{c_1}\cdots z_{i_r}^{c_r}- z_{j_1}^{b_1}\cdots z_{j_s}^{b_s}z_0^a=0\}$ is a divisor in $\PP^{N}$. 
 and it is a toric subvariety. Notice that the toric action is given by 
\[(\CC^*)^{N-1}\cong \{(\lambda_1,\cdots, \lambda_N)\in (\CC^*)^{N} 
||\lambda_{i_1}^{c_1}\cdots \lambda_{i_r}^{c_r}= \lambda_{j_1}^{b_1}\cdots \lambda_{j_s}^{b_s}\}.\] By intersecting all these divisors, we can obtain a toric subvariety.  

Notic that some of the equations are repeated in $(\CC^*)^N$, so in the following, we will define the variety only by those which is different equation in $(\CC^*)^N$, and the variety is the closeure of this.
\subsection{examples}
 \begin{exam}[$A_n$]
 Denote $p_0=(0,...,0)$, $p_i=e_i$ for $i=1,...,n$ and $p_{n+1}=(-1,...,-1)$. Then we have 
\[p_1+\cdots p_n=(0,...,0)=p_0,\] so the corresponding varieties, also denoted as  $A_n$, is given by
\[A_n=\{[z_0,\cdots, z_{n+1}]\in \PP^{n+1}| z_1\cdots z_{n+1}=z_0^{n+1}\}.\] 
\end{exam}
\begin{exam}[$D_n$]
Recall that $D_n:=Conv\{\pm e_i\}$. Denote $p_0=(0,...,0)$, $p_{2i-1}=e_i$, $p_{2i}=-e_i$ for $i=1,...,n$. then $p_{2i-1}+p_{2i}=0$ for $i=1,...,n$ gives n equations $z_{2i-1}z_{2i}=z_0^2$, and it gives a $n$ codimension subvariety of $\PP^{2n}$, hence these equations define $D_n$.
\[D_n=\{[z_0,...,z_{2n}]\in \PP^{2n}| z_{2i-1}z_{2i}=z_0^2\}.\]
\end{exam}
Given \[z_{i_1}^{c_1}\cdots z_{i_r}^{c_r}= z_{j_1}^{b_1}\cdots z_{j_s}^{b_s}z_0^a,\] we denote \[f(z_0,...,z_N)=z_{i_1}^{c_1}\cdots z_{i_r}^{c_r}- z_{j_1}^{b_1}\cdots z_{j_s}^{b_s}z_0^a.\]
\begin{exam}[$D(P)$]
Let $P$ is defined by \[\{[z_0,...,z_N]\in \PP^N| f_1=\cdots =f_r=0\}\] Then we define $\hat{f_i}(z_0,...,z_{N+2})=f_i(z_0,...,z_N)$. then by denote $p_{N+1}=(0,...,0,1)$, $p_{N+2}=(0,...,0,-1)$, we have a new equation 
\[z_{N+1}z_{N+2}=z_0^2.\]
Then the variety of $D(P)$ is given by 
 \[D(P)=\{[z_0,...,z_{N+1}z_{N+2}]\in \PP^N| f_1=\cdots =f_r=z_{N+1}z_{N+2}-z_0^2=0\}.\] 
\end{exam}
In order to define $D(P)$, we need to know $P$ as a subvariety of $\PP^N$, so in order to compute all the examples, we need to write down what  $X_i$ is as a subvariety. 
\begin{exam} As a subvariety of $\PP^i$,  restricted in $(\CC^*)^i\subset \PP^i$, $X_i$ are given by:
\begin{enumerate}
\item $X_3=A_2=\{[z_0,z_1,z_2,z_3]\in \PP^3| z_1z_2z_3=z_0^3\}$,
\item $X_4=D_2=\{[z_0,z_1,z_2,z_3,z_4]\in \PP^4| z_1z_2=z_0^2, z_3z_4=z_0^2\}$,
\item $X_6=\{[z_0,z_1,...,z_6]\in \PP^6| z_1z_4=z_2z_5=z_3z_6=z_0^2, z_2z_4=z_3z_0\}$, here the last equation comes from $(0,1)+(-1,0)=(-1,1)$, also, $z_2z_4=z_3z_0$ can be replaced by $z_1z_3z_5=z_0^3$ with other equations to get the same variety.
\item $X_8=\{[z_0,z_1,...,z_8]\in \PP^8|z_rz_{r+4}=z_0^2, \text{ where }r=1,2,3,4; z_1z_3=z_2z_0, z_3z_5=z_4z_0\}$, and
\item $X_9=\{[z_0,z_1,...,z_9]\in \PP^9|z_rz_{r+3}z_{r+6}=z_0^3, \text{ where } r=1,2,3; z_1z_3=z_2^2, z_2z_4=z_3^2, z_2z_6=z_0^2,z_3z_7=z_0^2\}$. Another way to write it is using $X_6$ plus 3 points, hence we need two more relation, namely 
\[\{[z_0,z_1,...,z_9]\in \PP^9|z_1z_4=z_2z_5=z_3z_6=z_0^2, z_2z_4=z_3z_0, z_7z_8z_9=z_0^3, z_7z_8=z_1z_2\}.\]
\end{enumerate}
With this , we can write $D(X_i)$ as a subvarieties of $\PP^{i+2}$. For example,
\begin{enumerate}
\item $D(X_3)=\{[z_0,z_1,z_2,z_3,z_4,z_5]\in \PP^5| z_1z_2z_3=z_0^3,z_4z_5=z_0^2\}$,
\item $D(X_4)=\{[z_0,z_1,...,z_6]\in \PP^6| z_1z_2=z_0^2, z_3z_4=z_0^2,z_5z_6=z_0^2\}$,
\item $D(X_6)=\{[z_0,z_1,...,z_8]\in \PP^8| z_1z_4=z_2z_5=z_3z_6=z_0^2z_2z_4=z_3z_0,z_7z_8=z_0^2\}$, 
\item $D(X_8)=\{[z_0,z_1,...,z_{10}]\in \PP^{10}|z_rz_{r+4}=z_0^2, \text{ where }r=1,2,3,4; z_1z_3=z_2z_0, z_3z_5=z_4z_0,z_9z_{10}=z_0^2\}$, and
\item $D(X_9)=\{[z_0,z_1,...,z_{11}]\in \PP^{11}|z_rz_{r+3}z_{r+6}=z_0^3,\text{ where } r=1,2,3; z_1z_3=z_2^2, z_2z_4=z_3^2, z_2z_6=z_0^2,z_3z_7=z_0^2,z_{10}z_{11}=z_0^2\}$. 
\end{enumerate} 
\end{exam}
We already know what $X_i\times \PP^1$ are, so we will only write down the equations of the varieties that we don't know what it is, in which we denoted as
\begin{exam} 
\begin{enumerate}
\item 
$P_1=Conv\{(\pm 1, 0, 0) (0,\pm 1, 0), (0, 0, .\pm 1), (-1,-1,-1),(1,1,1)\},$
which is $D(X_4)$ glue with two standard 3 simplex, which the faces are all standard 2 simplex,
\item $P_2=Conv\{ (\pm 1,0, 0),(0, \pm 1, 0), (\pm 1, \mp 1, 0), (0, 0, \pm 1), (\pm 1, 0, \mp 1), (0, \pm 1, \mp 1)\}$, 
\item $P_3$, which is given by the convex hull of the points:
\[(1,0,0),(1,1,0),(0,1,0),(-1,0,0),(-1,-1,0), (0,-1,0), \] \[(0,0,1),(1,0,1),(1,1,1), (0,1,1), (0,0,-1),(-1,0,-1),(-1,-1,-1), (0,-1,-1)\] which is a rhombic dodecahedron. Notice that $(0,1,1), (1,0,1), (0,-1,-1),(-1,0,-1)$.
\end{enumerate}
\begin{enumerate}
\item
$P_1=\overline{\{[z_0,z_1,...,z_8]\in (\CC^*)^8\subset \PP^8| z_1z_2=z_0^2, z_3z_4=z_0^2,z_5z_6=z_0^2,z_7z_8=z_0^2, z_1z_3z_5=z_7z_0^2\}}$, the last equation is deduced from $e_1+e_2+e_3=(1,1,1)$.
\item $P_2$ is a subvariety of $\PP^{12}$, so we need 9 equations. $z_{2r-1}z_{2r}=z_0^2$ for $r=1,...,6$, $z_1z_4=z_5z_0$, $z_7z_9=z_1z_0$ and $z_7z_{11}=z_3z_0$.
\item $P_3$ is a subvariety of $\PP^{14}$. Following the order above, we have 
$z_1z_4=z_2z_5=z_3=z_6=z_0^2$, $z_1z_3=z_2z_0$, $z_7z_9=z_8z_{10}$, $z_{11}z_{13}=z_{12}z_{14}$, $z_7z_{11}=z_8z_{12}=z_{9}z_{13}=z_{10}z_{14}=z_0^2$ and $z_8z_{11}=z_0z_1$.
\end{enumerate}
\end{exam}

\end{document}